\documentclass{amsart}
\usepackage[utf8]{inputenc}
\usepackage[nobysame,abbrev,alphabetic]{amsrefs}
\usepackage{amsmath}
\usepackage{amssymb}
\usepackage{xcolor}
\usepackage{constants}
\usepackage{enumitem} 
\usepackage{hyperref}
\usepackage{mathtools}
\usepackage{changepage}

\allowdisplaybreaks[4]

\newtheorem{thm}{Theorem}[section]
\newtheorem{lem}[thm]{Lemma}
\newtheorem{prop}[thm]{Proposition}
\newtheorem{cor}[thm]{Corollary}
\newtheorem{defn}[thm]{Definition}

\newcommand{\R}{\mathbb{R}}
\newcommand{\N}{\mathbb{N}}

\newcommand{\Sph}{{\mathbb S}}

\newcommand{\bc}{\begin{cor}}
\newcommand{\ec}{\end{cor}}

\newcommand{\bl}{\begin{lem}}
\newcommand{\el}{\end{lem}}

\newcommand{\bp}{\begin{prop}}
\newcommand{\ep}{\end{prop}}

\newcommand{\bd}{\begin{defn}}    
\newcommand{\ed}{\end{defn}}

\newtheorem{rmrk}[thm]{Remark}   

\newcommand{\br}{\begin{rmrk}}
\newcommand{\er}{\end{rmrk}}

\newcommand{\be}{\begin{equation}}

 \newcommand{\ee}{\end{equation}}

\DeclareMathOperator{\Vol}{Vol}
\DeclareMathOperator{\vol}{Vol}
\DeclareMathOperator{\Area}{Area}

\newcommand{\Sc}{{\rm Sc}} 

\def\Xint#1{\mathchoice
{\XXint\displaystyle\textstyle{#1}}%
{\XXint\textstyle\scriptstyle{#1}}%
{\XXint\scriptstyle\scriptscriptstyle{#1}}%
{\XXint\scriptscriptstyle\scriptscriptstyle{#1}}%
\!\int}
\def\XXint#1#2#3{{\setbox0=\hbox{$#1{#2#3}{\int}$ }
\vcenter{\hbox{$#2#3$ }}\kern-.6\wd0}}

\def\dashint{\Xint-}

\title[Conformal Scalar Curvature Compactness Conjecture]{On the Scalar Curvature Compactness Conjecture in the Conformal Case}

\author{Brian Allen}
\address[Brian Allen]{Lehman College, CUNY}
\email{brianallenmath@gmail.com}

\author{Wenchuan Tian}
\address[Wenchuan Tian]{University of California, Santa Barbara}
\email{tian.wenchuan@gmail.com}

\author{Changliang Wang}
\address[Changliang Wang]{
School of Mathematical Sciences and Institute for Advanced Study, Key Laboratory of Intelligent Computing and Applications(Ministry of Education), Tongji University, Shanghai 200092, China}
\email{wangchl@tongji.edu.cn}

\begin{document}

\maketitle
    
\begin{abstract}
Is a sequence of Riemannian manifolds with positive scalar curvature, satisfying some conditions to keep the sequence reasonable, compact? What topology should one use for the convergence and what is the regularity of the limit space? In this paper we explore these questions by studying the case of a sequence of Riemannian manifolds which are conformal to the $n$-dimensional round sphere. We are able to show that the sequence of conformal factors are compact in several analytic senses and are able to establish $C^0$ convergence away from a singular set of small volume in a similar fashion as C. Dong \cite{Dong-PMT_Stability}. Under a bound on the total scalar curvature we are able to show that the limit conformal factor has  weak positive scalar curvature in the sense of weakly solving the conformal positive scalar curvature equation.
\end{abstract}
    
\section{Introduction}

It is an important open question to better understand Riemannian manifolds with scalar curvature bounded from below. Towards this goal, many important rigidity theorems involving scalar curvature have been established such as the positive mass theorem, the Geroch conjecture, and Llarull's theorem. In order to further understand lower curvature bounds we can explore geometric stability questions where we ask for the conditions of the rigidity theorem to almost be satisfied for a sequence and wonder in which sense the sequence converges to the rigid Riemannian manifold. One approach to these geometric stability theorems is to first apply a compactness theorem for sequences with lower scalar curvature bounds and then prove the rigidity theorem in the low regularity setting suggested by the compactness theorem. For instance, if one considers a sequence of tori with scalar curvature $\Sc_j\ge - \frac{1}{j}$, satisfying a reasonable set of extra conditions to keep the sequence from degenerating in uninteresting ways, we can prove stability by first showing the sequence is compact in some particular class of metric spaces and then prove rigidity of the flat torus in this class. Additionally, a resolution to the compactness question for Riemannian manifolds with positive scalar curvature would undoubtedly produce insights into scalar curvature even without applying the compactness theorem to prove stability of scalar curvature rigidity theorems. 

One such compactness conjecture in dimension three has been suggested by M. Gromov \cite{Gromov-Plateau} and further refined by C. Sormani \cite{SormaniIAS} where the conjecture says that the sequence should converge in the volume preserving Sormani-Wenger Intrinsic Flat sense to a rectifiable geodesic metric space with Euclidean tangent cones almost everywhere and satisfying some weak notion of positive scalar curvature. The rotationally symmetric case of this conjecture was studied by J. Park, W. Tian, and C. Wang \cite{Park-Tian-Wang-18} where they were able to give a full answer to the conjecture posed by C. Sormani. More recently, warped products on $S^2\times S^1$ have been studied by W. Tian and C. Wang \cite{tian2023compactness} with an important related example explored by C. Sormani, W. Tian and C. Wang \cite{sormani2023extreme} and open questions announced by C. Sormani \cite{SormaniOberwolfach}. See D. Kazaras and K. Xu \cite{KK} for more important examples related to the conjecture.

In this paper we explore the scalar curvature compactness theorem in the case of a sequence of Riemannian manifolds which are conformal to the $n$-dimensional round sphere. Where the conjecture of C. Sormani assumes a lower bound on the area of all closed minimal surfaces in the Riemannian manifold, we assume a uniform integrability condition on the volume forms $e^{nf_j}$ in \eqref{UnifIntegrabilityVolume} which was also assumed by B. Allen \cite{Allen-conf-torus} in the conformal setting of the Geroch stability conjecture. Studying the compactness of conformal metrics under curvature bounds is also an interesting question in its own right and has been studied by A. Chang and P. Yang \cite{Chang-Yang} on $\Sph^3$ in the case of isospectral metrics, M. Gursky \cite{Gursky} with a $L^p$, $p>\frac{n}{2}$ bound on the Riemann tensor, Y. Li and Z. Zhou \cite{Li-Zhou} and C. Dong, Y. Li, and K. Xu \cite{DLX} with a $L^p$, $p>\frac{n}{2}$ bound on the scalar curvature, and C. Dong and Y. Li \cite{Dong-Li} with a $L^p$, $p>\frac{n}{2}$ bound on the Ricci tensor. In all of these papers the authors are able to conclude stronger notions of convergence due to the stronger curvature assumptions assumed.

In this paper we are able to show that the sequence of conformal factors are compact in several analytic senses and are able to establish $C^0$ convergence away from a singular set of small volume and boundary area in a similar fashion as C. Dong \cite{Dong-PMT_Stability}. The method of C. Dong \cite{Dong-PMT_Stability} has become popular for studying scalar curvature stability problems and has been adapted and applied by C. Dong and A. Song \cite{Dong-Song}, B. Allen, E. Bryden, and D. Kazaras \cite{ABKLLarull}, S. Hirsch and Y. Zhang \cite{HZ}, and more recently by C. Dong \cite{dong2024stability}. Our first main result in this direction is stated as follows.

\begin{thm}\label{thm-MainTheorem1}
    Let $V, \Lambda > 0$, $ n\geq 3$,  $(\Sph^n,g_{\Sph^n})$ be the standard round sphere, and $f_j:\Sph^n\rightarrow \R$ a sequence of smooth functions defining conformal metrics $g_j=e^{2f_j}g_{\Sph^n}$. If 
    \begin{align}
        \Sc_{g_j} &\ge 0, \qquad V^{-1} \le \Vol_{g_j}(\Sph^n) \le V, \qquad
        \\\Vol_{g_j}(U) &\le  \Lambda \Vol_{g_{\Sph^n}}(U)^{\alpha}, \alpha \in (0,1), \quad \forall U \subset \Sph^n \text{ measurable}, \label{UnifIntegrabilityVolume}
    \end{align}
    then there exists a subsequence, and a limiting  function $e^{\frac{(n-2)f_{\infty}}{2}} \in W^{1,q}$ for $1\le q< \frac{4n}{3n-2}$ which is bounded away from zero so that $e^{\frac{(n-2)f_j}{2}}$ converges to $e^{\frac{(n-2)f_{\infty}}{2}}$ in the $L^p$ sense for $1\le p < \frac{2n}{n-2}$,  and $e^{\frac{(n-2)f_j}{2}}$ converges weakly in $W^{1,p}$ for $p < \frac{4n}{3n-2}$. We also have that $e^{\frac{(n-2)f_j}{2}}$ is uniformly bounded away from zero, $e^{\frac{(n-2)f_{\infty}}{2}} \in (0,\infty]$ is well defined everywhere, and $e^{\frac{(n-2)f_{\infty}}{2}}$ is lower semicontinuous.
    
    Moreover, $f_j$ is uniformly bounded from below, $f_{\infty} \in (-\infty,\infty]$ is defined everywhere, $f_{\infty}$ is lower semicontinuous, $f_j \rightarrow f_{\infty}$ in $L^p$, $p \in \left[1,\frac{2n}{n-2}\right)$, and there is a measurable set $Z_j \subset \Sph^n$ where if we let $g_\infty:=e^{2 f_{\infty}}g_{\Sph^n}$ then
    \begin{align}
    \Area_{g_{\Sph^n}}(\partial Z_j)&+\Vol_{g_{\Sph^n}}(Z_j)\rightarrow 0,
        \\\Vol_{g_j}(\Sph^n \setminus Z_j) &\rightarrow \Vol_{g_{\infty}}(\Sph^n), 
        \\|f_j-f_{\infty}|^2&\le C_j \text{ on } \Sph^n \setminus Z_j, \quad C_j \searrow 0. 
    \end{align}

\end{thm}
\begin{rmrk}
{\rm
    Notice that if $e^{f_{\infty}}$ is bounded above then lower semicontinuity implies that $(\Sph^n,e^{2f_{\infty}}g_{\Sph^n})$ is a well defined length space. See the discussion in the beginning of section \ref{sec:Spherical Mean Method} for further details.
}
\end{rmrk}

Under a bound on the total scalar curvature we are able to show stronger notions of convergence and that the limiting conformal factor is bounded away from zero and has weak positive scalar curvature in the  sense of Definition \ref{def-Weak Positive Scalar Confromal} which implies that the conformal factor weakly solves the positive scalar curvature equation. 
\begin{thm}\label{thm-MainTheorem2}
    Let $V, \Lambda, R_0 > 0$, $ n\geq 3$,  $(\Sph^n,g_{\Sph^n})$ be the standard round sphere, and $f_j:\Sph^n\rightarrow \R$ a sequence of functions defining conformal metrics $g_j=e^{2f_j}g_{\Sph^n}$. If 
    \begin{align}
        \Sc_{g_j} &\ge 0, \qquad \int_{\Sph^n} \Sc_{g_j} dV_{g_j} \le R_0, \qquad V^{-1} \le \Vol_{g_j}(\Sph^n) \le V, \qquad
        \\\Vol_{g_j}(U) &\le  \Lambda \Vol_{g_0}(U)^{\alpha}, \alpha \in (0,1), \quad \forall U \subset \Sph^n \text{ measurable}, \label{UnifIntegrabilityVolume}
    \end{align}
    then there exists a subsequence, and a limiting  function $e^{\frac{(n-2)f_{\infty}}{2}} \in W^{1,q}$ for $1\le q< \frac{4n}{3n-2}$ which is bounded away from zero. 
   Furthermore, one obtains weak $W^{1,2}$ convergence of $e^{\frac{(n-2)f_j}{2}}$ to $e^{\frac{(n-2)f_{\infty}}{2}}$, $Z_j$ can be chosen as in Theorem \ref{thm-MainTheorem1} with the stronger condition that $\left|e^{\frac{(n-2)f_j}{2}}-e^{\frac{(n-2)f_{\infty}}{2}}\right|^2 \le C_j$, $C_j \searrow 0$, on $\Sph^n \setminus Z_j$, and $(\Sph^n,e^{2f_{\infty}}g_{\Sph^n})$ has weak positive scalar curvature in the sense of Definition \ref{def-Weak Positive Scalar Confromal}.
\end{thm}

 We are able to show that the limiting function $e^{\frac{(n-2)f_\infty}{2}}$ is either positive or zero by looking at truncated solutions which weakly solve an elliptic partial differential equation and applying the $W^{1,2}$ maximum principle. Then we are able to use the volume lower bound and the uniform integrability of volume \eqref{UnifIntegrabilityVolume} to show that the limiting function cannot be identically zero in subsection \ref{subsect: positive lower bound}. In \cite{Allen-conf-torus}, the conformal case of stability of the Geroch conjecture was studied and the uniform integrability of the volume was also implemented as an assumption to remove the possibility of bubbling. In the conformal case, bubbling appears when volume is allowed to concentrate at a point, causing the sequence of volume forms, $e^{nf_j}$, to not be compact in $L^1$. By adding the uniform integrability condition for volume \eqref{UnifIntegrabilityVolume} we obtain a uniform integrability condition for the volume forms, $e^{nf_j}$, which implies that they are compact in $L^1$ and that volume cannot concentrate at a point. It is important to note that when volume concentrates at a point, the geometry of the limiting metric space is two different Riemmanian manifolds attached at the blowup point, as has been explicitly studied in Example 3.5 by B. Allen and C. Sormani \cite{AS-relating}.

 Then the positive infimum of the limiting function $e^{\frac{(n-2)f_\infty}{2}}$ is obtained in Proposition \ref{prop: positive lower bound}, providing a uniform positive lower bound for the sequence of conformal factors $e^{\frac{(n-2)f_j}{2}}$ in Proposition \ref{prop: uniform positive lower bound}, with the help of the spherical mean inequalities established in section \ref{sec:Spherical Mean Method}. While the spherical mean inequality for the conformal factor $u$ in Lemma \ref{lem: spherical mean inequality} can only be established in dimensions $3 \leq n \leq 4$ due to an elementary inequality (see Remark \ref{rmrk: elementary inequality}), the spherical mean inequality for $u^{\frac{2}{n-2}}$ can be derived in dimensions of $n \geq 5$ in Lemma \ref{lem: spherical mean inequality high dim}. By Lemma \ref{Lem-Superharmonic}, note that $5$ is the lowest dimension in which the spherical mean inequality for $u^{\frac{2}{n-2}}$ can be established. Moreover, from these spherical mean inequalities, by using a similar argument as in the proof of the Bishop-Gromov volume comparison theorem, the ball average monotonicity properties are obtained in Lemmas \ref{LemmaBishopGromov} and \ref{lem: ball average inequality high dim}. Consequently, we are able to show that the limiting function $e^{\frac{(n-2)f_\infty}{2}}$ is lower semi-continuous in Proposition \ref{prop: lower semi-continuity of u high dim}. The spherical mean argument is also obtained for the sequence of logs of the conformal factors $f_j = \frac{2}{n-2}\ln u_j$ in subsection \ref{subsect: spherical mean inequality for f}, and we obtain semi-continuity of $f_\infty$ as stated in Theorem \ref{thm-MainTheorem1}.

In Theorem \ref{thm-MainTheorem1} we see that without additional assumptions we are only able to show that the log of the conformal factors converge uniformly outside a singular set and we are unable to show that the limit has positive scalar curvature in a weak sense. When we add that the total scalar curvature is bounded in Theorem \ref{thm-MainTheorem2} then we are able to conclude weak $W^{1,2}$ convergence of $e^{\frac{(n-2)f_j}{2}}$ which is enough to show uniform convergence of this power of the conformal factor away from a singular set of small volume and boundary area. Furthermore we are able to show that the limiting conformal metric 
has positive scalar curvature in the sense of weakly solving the conformal positive scalar curvature equation defined in Definition \ref{def-Weak Positive Scalar Confromal}. 
In the case where the limiting function is bounded, the lower semi-continuity of the limiting function in Proposition \ref{prop: lower semi-continuity of u high dim} implies that the limiting conformal metric defines a $C^0$ Riemannian manifold with weak positive scalar curvature.

\noindent
{\bf Acknowledgements:} 

Changliang Wang was partially supported by the Fundamental Research Funds for the Central Universities and Shanghai Pilot Program for Basic Research.


\section{Regularity of Solutions} \label{sec: Regularity of Solutions}

We start by reviewing the formula for scalar curvature under a conformal change as well as develop our first regularity consequences.

\begin{lem}\label{lem:ScalarFormulas}
    Let $(\Sph^n,g_{\Sph^n})$ be the round sphere. If the conformal metric $g=e^{2f}g_{\Sph^n}$ has nonnegative scalar curvature, i.e. \begin{align}\label{eq:ConformalScalarFormula}
        \Sc_{g} = e^{-2f} \left(\Sc_{g_{\Sph^n}}-2(n-1)\Delta f - (n-2)(n-1)|\nabla f|^2\right)  \geq 0,
    \end{align}
    where $f$ is a smooth function on $\Sph^n$,
    then $f$ satisfies \begin{align}\label{eq:ScalarLemmaFormula}
      \int_{\Sph^n}|\nabla f|^2dV_{g_{\Sph^n}}&\le \frac{n\Vol_{g_{\Sph^n}}(\Sph^n)}{n-2}.
    \end{align}
\end{lem}
\begin{proof}
Recall the scalar curvature formula of conformally changed metric as:    \begin{align}\label{eq:ScalarLemmaeq1}
        \Sc_{g} = e^{-2f} \left(\Sc_{g_{\Sph^n}}-2(n-1)\Delta f - (n-2)(n-1)|\nabla f|^2\right) &\ge 0, 
    \end{align}
    and so by rearranging \eqref{eq:ScalarLemmaeq1} we see
    \begin{align}\label{eq:ScalarLemmaeq2}
      2\Delta f +(n-2)|\nabla f|^2&\le \frac{\Sc_{g_{\Sph^n}}}{n-1}.
    \end{align}
    Now by integrating \eqref{eq:ScalarLemmaeq2} we find
     \begin{align}\label{eq:ScalarLemmaeq3}
      \int_{\Sph^n}|\nabla f|^2dV_{g_{\Sph^n}}&\le \int_{\Sph^n}\frac{\Sc_{g_{\Sph^n}}}{(n-1)(n-2)}dV_{g_{\Sph^n}} = \frac{n \Vol(\Sph^n)}{n-2}.
    \end{align}
\end{proof}

Now our goal is to improve on the regularity of the results in Lemma \ref{lem:ScalarFormulas}. 
\bl\label{lem: gradient L2 bound of f}
 Let $(\Sph^n,g_{\Sph^n})$ be the standard round sphere. Assume that $g :=e^{2f}g_{\Sph^n}$ is a conformal Riemannian metric on $\Sph^n$ with nonnegative scalar curvature, such that the volume $\vol(\Sph^n, g)\leq V$ for some $V>0$. Then for any $p\in [0, \frac{n-2}{2})$
\be
\int_{\Sph^n} e^{pf} |\nabla f|^2dV_{g_{\Sph^n}}\leq \frac{nV^{p/n} \Vol(\Sph^n)^{(n-p)/n}}{n-2-2p}.
\ee
\el

\begin{proof}
Since 
\be
e^{2f} \Sc_g =\Sc_{g_{\Sph^n}}-2(n-1)\Delta f-(n-1)(n-2)|\nabla f|^2,
\ee
with $\Sc_{g}\geq 0$ and $\Sc_{g_{\Sph^n}}=n(n-1)$, we have the inequality:

\be
n(n-1) \geq 2(n-1)\Delta f+(n-1)(n-2)|\nabla f|^2.
\ee
For any $p\in \left[0, \frac{n-2}{2}\right)$, multiply the inequality by $e^{pf}$, then, after integration by parts we get
\be
\int_{\Sph^n}n e^{pf}dV_{g_{\Sph^n}}\geq (n-2-2p) \int_{\Sph^n} e^{pf}|\nabla f|^2dV_{g_{\Sph^n}}.
\ee
By H\"older's inequality we have
\be
\int_{\Sph^n} e^{pf}dV_{g_{\Sph^n}}\leq \left(\int_{\Sph^n} e^{nf} dV_{g_{\Sph^n}}\right)^{p/n} \Vol(\Sph^n)^{(n-p)/n}.
\ee
Combine these inequalities then we have the desired result.
\end{proof}

The previous lemma leads to another important regularity result for a power of the conformal factor.

\begin{lem}\label{lem: improved weak bound}
 Let $(\Sph^n,g_{\Sph^n})$ be the round sphere. Let $g:=e^{2f}g_{\Sph^n}$ be a metric on $\Sph^n$ with nonnegative scalar curvature, such that $\vol(\Sph^n, g)\leq V$ for some $V>0$. Then for any $q \in \left[1, \frac{n}{n-1} \right]$, there exists a constant $C(g_{\Sph^n}, q)$ such that
\begin{align}
    \int_{\Sph^n} |\nabla e^{(n-2)f/2}|^q dV_{g_{\Sph^n}} \le C(g_{\Sph^n}, q).
\end{align}
\end{lem}
\begin{proof}
For any $q \in \left[1, \frac{n}{n-1}\right]$, by H\"older inequality, we have
\begin{eqnarray}
\int_{\Sph^n} \left| \nabla \left( e^{\frac{n-2}{2}f}\right) \right|^q dV_{g_{\Sph^n}}
& = & \int_{\Sph^n} e^{q \frac{n-2}{2}f} |\nabla f|^q dV_{g_{\Sph^n}}\\
& \leq & \left( \int_{\Sph^n} |\nabla f|^2 dV_{g_{\Sph^n}} \right)^{\frac{q}{2}} \left( \int_{\Sph^n} e^{\frac{q (n-2)}{2-q}f} dV_{g_{\Sph^n}} \right)^{\frac{2-q}{2}}.
\end{eqnarray}
Note that the estimate in Lemma \ref{lem:ScalarFormulas} provides a upper bound for the first factor. Because for $q \in \left[1, \frac{n}{n-1} \right]$, $\frac{q(n-2)}{2-q} \leq n$, the assumption $\vol(\Sph^n, g) \leq V$ produces a upper bound for the second factor. This completes the proof.
\end{proof}

With a bit more effort we can improve on the regularity obtained in the previous lemma.

\begin{prop}\label{prop:Improved weak bound}
 Let $(\Sph^n,g_{\Sph^n})$ be the round sphere. Let $g:=e^{2f}g_{\Sph^n}$ be a metric on $\Sph^n$ with non-negative scalar curvature, such that $\vol(\Sph^n, g)\leq V$ for some $V>0$. Then for any $q \in \left[1, \frac{4n}{3n-2} \right)$, there exists a constant $C(g_{\Sph^n}, V, q)$ such that
\begin{align}
    \int_{\Sph^n} |\nabla e^{(n-2)f/2}|^q dV_{g_{\Sph^n}} \le C(g_{\Sph^n}, V, q).
\end{align}
\end{prop}
\begin{proof}
By Lemma \ref{lem: improved weak bound}, it suffices to derive the estimate for $q \in \left( \frac{n}{n-1}, \frac{4n}{3n-2} \right)$. Now for any such $q$ and for any $\theta \in (0,1/2)$ we have
\be
e^{q(n-2)f/2}|\nabla f|^q=e^{q(1-\theta)(n-2)f/2} e^{q\theta(n-2)f/2}|\nabla f|^q.
\ee
By H\"older's inequality, choose $t= 2/q$ and $t'= 2/(2-q)$, then we have
\begin{eqnarray}
& & \int_{\Sph^n}e^{q(n-2)f/2}|\nabla f|^qdV_{g_{\Sph^n}}  \\
& = & \int_{\Sph^n}e^{q(1-\theta)(n-2)f/2} e^{q\theta(n-2)f/2}|\nabla f|^qdV_{g_{\Sph^n}}\\
& \leq & \left( \int_{\Sph^n}e^{\frac{q(n-2)}{2-q}(1-\theta)f}dV_{g_{\Sph^n}} \right)^{\frac{2-q}{2}} \left( \int_{\Sph^n}e^{\theta(n-2)f}|\nabla f|^2dV_{g_{\Sph^n}} \right)^{\frac{q}{2}}.
\end{eqnarray}
Set $\frac{q(n-2)}{2-q}(1-\theta)=n$ and solve for $q$, then we have
\be
q=\frac{2n}{n+(n-2)(1-\theta)}.
\ee
Note that here $q$ is an increasing function of $\theta$ and as $\theta \to 1/2$ we have $q\to \frac{4n}{3n-2}$, and as $\theta \to 0$ we have $q \to \frac{n}{n-1}$. The assumption $\vol(\Sph^n, g) \leq V$ produces a upper bound for the first factor, and Lemma \ref{lem: gradient L2 bound of f} provides a upper bound for the second factor.
\end{proof}

As an immediate consequence we can obtain our first convergence result.

\begin{cor}\label{cor:W12WeakConvergence}
     Let $(\Sph^n,g_{\Sph^n})$ be the round sphere.  If $g_{j}=e^{2f_j}g_{\Sph^n}$, $f_j:\Sph^n \to \R$, $\Sc_{g_{j}} \ge 0$, and $\vol_{g_j}(\Sph^n) = \int_{\Sph^n}e^{nf_j}dV_{\Sph^n}\le C$, then for any $q \in \left[1, \frac{4n}{3n-2} \right)$,
    \begin{align}
        e^{\frac{(n-2)f_j}{2}} &\rightharpoonup  e^{\frac{(n-2)f_{\infty}}{2}}, \quad \text{ in }W^{1,q}
    \end{align}
    for some $e^{\frac{(n-2)f_{\infty}}{2}} \in W^{1,q}(\Sph^n,g_{\Sph^n})$ and
    \begin{align}
       e^{\frac{(n-2)f_j}{2}} &\rightarrow e^{\frac{(n-2)f_{\infty}}{2}}, \quad \text{ in }L^{p}, 
    \end{align}
    for $1 \le p < \frac{2n}{n-2}$.
\end{cor}
\begin{proof}
    By Proposition \ref{prop:Improved weak bound} we know that for $q \in \left(1, \frac{4n}{3n-2} \right]$,
   \begin{align}
    \int_{\Sph^n} |\nabla e^{(n-2)f/2}|^q dV_{g_{\Sph^n}} \le C(g_{\Sph^n}, V, q).
\end{align}
    Now we calculate
    \begin{align}\label{eq:Intermediate Lp Bound}
        \int_{\Sph^n} \left(e^{\frac{(n-2)f_j}{2}}\right)^{\frac{4n}{3n-2}}dV_{g_{\Sph^n} }=\int_{\Sph^n} e^{\frac{2n(n-2)f_j}{3n-2}}dV_{g_{\Sph^n} },
    \end{align}
    and notice that 
    \begin{align}
      \frac{2n(n-2)}{3n-2} = n \frac{2n-4}{3n-2}<n,  
    \end{align}
    and so by assumption and H\"{o}lder's inequality we know that \eqref{eq:Intermediate Lp Bound} is bounded. This implies that a subsequence converges weakly in $W^{1,q}(\Sph^n,g_{\Sph^n} )$. By the Rellich-Kondrachov compactness theorem this implies convergence in $L^p$ for $\frac{4n}{3n-6}>p>1$. By the assumption that $e^{f_j}$ is bounded in $L^n$ we can use interpolation to find convergence in $L^p$ for $1\le p < \frac{2n}{n-2}$.
\end{proof}

We now investigate the notion of weak positive scalar curvature which is natural in the conformal case. In Lemma \ref{lem:ScalarFormulas} we observed that for a conformal Riemannian manifold $(\Sph^n,g_f=e^{2f}g_{\Sph^n})$, the scalar curvature formula yields
\be
e^{2f} \Sc_{g_f}=\Sc_{g_{\Sph^n}}-2(n-1)\Delta f-(n-2)(n-1)|\nabla f|^2.
\ee

If we take $e^{2f}=u^{\frac{4}{n-2}}$, or equivalently $u=e^{\frac{(n-2)f}{2}}$, then $u$ satisfies the equation
\be
-\frac{4(n-1)}{n-2}\Delta u+\Sc_{g_{\Sph^n}} u=\Sc_{g_f}u^{\frac{n+2}{n-2}}.
\ee
The scalar curvature $R_{g_f} \geq 0$ implies the inequality
\be\label{eq: Weak Poisitive Scalar Curvature}
-\frac{4(n-1)}{n-2}\Delta u+\Sc_{g_{\Sph^n}} u\geq 0.
\ee

This motivates a definition of weak scalar curvature in the conformal case.

\begin{defn}\label{def-Weak Positive Scalar Confromal}
    Assume $(M,g_0)$ is a Riemannian manifold. We say that $(M,e^{2f}g_0)$ has weak positive scalar curvature if $u=e^{\frac{(n-2)f}{2}} \in W^{1,2}(M,g_0)$ weakly solves the  inequality
    \be\label{eq: Weak Poisitive Scalar Curvature}
-\frac{4(n-1)}{n-2}\Delta u+\Sc_{g_0} u\geq 0,
\ee
which implies that for all nonnegative test functions $\varphi \in W^{1,2}(M,g_0)$ we have
\begin{align}
    - \int_M g_0(\nabla \varphi , \nabla u)dV_{g_0} \le \frac{n-2}{4(n-1)} \int_M \Sc_{g_0} u \varphi dV_{g_0}.
\end{align}
\end{defn}

One can check that this natural notion of weak positive scalar curvature for conformal Riemannian manifolds is not quite the same as the notion defined by Lee-LeFloch \cite{Lee-LeFloch} for general Riemmanian manifolds which we are unable to establish given the assumptions in our setting. 



\section{The Spherical Mean Method}\label{sec:Spherical Mean Method}

In this section we are able to apply the spherical mean method for the round sphere in order to extract even more information about the limiting conformal factor. In particular, we are able to show that the limiting conformal factor has to be lower semicontinuous and defined everywhere. Being lower semicontinuous and defined everywhere is important when combined with boundedness since it implies that the limiting conformal factor defines a well defined length space. To see this, first notice that if $u_{\infty}$ is lower semiconituous then so is $u_{\infty}^{\frac{2}{n-2}}$. Then notice that for any piecewise smooth curve $\gamma:[a,b]\rightarrow \Sph^n$, $u_{\infty}(\gamma)^{\frac{2}{n-2}}$ is also lower semicontinuous and hence measurable. Then since $u_{\infty}$ is assumed to be bounded we know that $u_{\infty}(\gamma)^{\frac{2}{n-2}}$ is integrable and hence the length of every curve is defined.  

Due to establishment of an elementary estimate in Lemma \ref{lem: elementary inequality}, we have to apply the spherical mean method on $\Sph^n$ for $3 \leq n \leq 4$ and $n \geq 5$ separately, see more explanations in Remark \ref{rmrk: elementary inequality}. In the case of the three-sphere and four-sphere, we prove a spherical mean inequality for the conformal factor $u$ in Subsection \ref{subsect: spherical mean lower dim}. In the case of higher dimensional spheres, we prove a spherical mean inequality for $u^{\frac{2}{n-2}}$ in Subsection \ref{subsect: spherical mean inequality high dimensions}. In Subsection \ref{subsect: spherical mean inequality for f}, we apply spherical mean method for $f = \frac{2}{n-2} \ln u$.

We now review some notation used throughout this section.
We use $B_r(x)$ to denote the geodesic ball in the round sphere $\Sph^n$ with radius $r$ and centered at $x$ and $\partial B_r(x)$ to denote its boundary. We use $d\sigma$ to denote the volume form on $\partial B_r(x)$ induced from the round metric on $\Sph^n$. Then in polar coordinates on $\Sph^n$ we have $g_{\Sph^n}=dr^2 +\sin^2 r g_{\Sph^{n-1}}$, $\Vol_{\Sph^n}(B_r(x))=\omega_{n-1} \int^r_0 \sin^{n-1}s ds$, and $d\sigma= \sin^{n-1} r dV_{g_{\Sph^{n-1}}}$. Recall that $\omega_n := \frac{2\pi^{\frac{n+1}{2}}}{\Gamma\left( \frac{n+1}{2} \right)}$ is the volume of $\Sph^n$ with the standard metric $g_{\Sph^n}$. 



\subsection{Spherical mean inequality for $u$ on the $n$-sphere ($n = 3, 4$)}\label{subsect: spherical mean lower dim}

In this subsection we apply the spherical mean method to $u$ which will yield interesting results in dimensions $n=3,4$. We start with a standard calculation for the spherical mean in terms of the Laplacian.

\bl\label{LemmaPhiDerivative}
For any fixed $x \in \Sph^n$, define
\be
\phi(r):= \dashint_{\partial B_r(x)} u d\sigma.
\ee
Then the derivative of $\phi$ with respect to $r$ is given by
\be
\phi'(r)=\frac{ \int_{B_r(x)} \Delta u dV_{g_{\Sph^n}}}{\omega_{n-1} \sin^{n-1} r}
\ee
\el
\begin{proof}
From the definition we have
\be
\phi(r)=\frac{\sin^{n-1}r \int_{\Sph^{n-1}} u(r, \theta) dV_{g_{\Sph^{n-1}}}(\theta)}{\omega_{n-1} \sin^{n-1}r}=\frac{\int_{\Sph^{n-1}} u(r, \theta) dV_{g_{\Sph^{n-1}}}(\theta)}{\omega_{n-1}}.
\ee
We can take derivative to get
\be
\begin{split}
\phi'(r) &=\frac{\int_{\Sph^{n-1}} \frac{\partial u}{\partial r}(r, \theta) dV_{g_{\Sph^{n-1}}}}{\omega_{n-1}}\\
&=\frac{\sin^{n-1}r \int_{\Sph^{n-1}} \frac{\partial u}{\partial r}(r, \theta) dV_{g_{\Sph^{n-1}}}}{\omega_{n-1} \sin^{n-1}r}\\
&=\frac{\int_{\partial B_r(x)} g_{\Sph^n}( \nabla u, \partial_r)  d\sigma}{\omega_{n-1} \sin^{n-1}r}\\
&=\frac{\int_{B_r(x)} \Delta u dV_{g_{\Sph^n}}}{\omega_{n-1} \sin^{n-1}r}.
\end{split}
\ee
where the last step follows from Stokes' theorem.
\end{proof}

We now prove an elementary lemma which shows us where the spherical mean inequality for $u$ breaks down in dimensions $5$ and higher.
\bl\label{LemmaCalculus}
For $r\in (0,\frac{\pi}{2})$, and $n \in \{3, 4\}$,
\be
\frac{\left( \int^{r}_0 \sin^{n-1}s ds \right)^{\frac{n+2}{2n}}}{\sin^{n-1} r} \leq \left( \frac{\pi}{2} \right)^{\frac{n+2}{2n}},
\ee
\el
\begin{proof}
By the mean value property we have
\be
\int^r_0 \sin^{n-1}s ds= r\sin^{n-1} \xi
\ee
for some $\xi\in (0,r)$. As a result, we have
\be
\frac{\left( \int^{r}_0 \sin^{n-1}s ds \right)^{\frac{n+2}{2n}}}{\sin^{n-1} r}
= \frac{\left( r \sin^{n-1} \xi \right)^{\frac{n+2}{2n}}}{\sin^{n-1}r} \leq 
\frac{\left( r \sin^{n-1} r \right)^{\frac{n+2}{2n}}}{\sin^{n-1}r}
\ee
For $n=3$, 
\be
\frac{\left( r \sin^{n-1} r \right)^{\frac{n+2}{2n}}}{\sin^{n-1}r}
=
\frac{\left( r \sin^2 r \right)^{\frac{5}{6}}}{\sin^{2}r} 
= \frac{r^{\frac{5}{6}}}{\sin^{\frac{1}{3}}r} = r^{\frac{1}{2}} \left( \frac{r}{\sin r}\right)^{\frac{1}{3}} \leq \left( \frac{\pi}{2} \right)^{\frac{5}{6}}.
\ee
The last inequality follows from the fact that $\frac{r}{\sin r}\leq \pi/2$ for $r\in (0,\pi/2)$. We can see this by taking derivative to see that $\frac{r}{\sin r}$ is increasing for $r\in (0,\pi/2)$. Then the conclusion for $n=3$ follows.

For $n=4$,
\be
\frac{\left( r \sin^{n-1} r \right)^{\frac{n+2}{2n}}}{\sin^{n-1}r}
= \frac{\left( r \sin^{3}r \right)^{\frac{3}{4}}}{\sin^{3}r}
= \left( \frac{r}{\sin r} \right)^{\frac{3}{4}} \leq \left( \frac{\pi}{2} \right)^{\frac{3}{4}}.
\ee
From this estimate the conclusion for $n=4$ follows.
 
\end{proof}
\begin{rmrk}\label{rmrk: elementary inequality}
{\rm
    Notice that for general $n$ we have
    \begin{align}
\frac{\left( \int^{r}_0 \sin^{n-1}s ds \right)^{\frac{n+2}{2n}}}{\sin^{n-1} r}
&= \frac{\left( r \sin^{n-1} \xi \right)^{\frac{n+2}{2n}}}{\sin^{n-1}r} 
\\&\leq 
\frac{\left( r \sin^{n-1} r \right)^{\frac{n+2}{2n}}}{\sin^{n-1}r}
\\&= 
\frac{ r^{\frac{n+2}{2n}}}{\sin^{\frac{n^2-3n+2}{2n}}r}= 
r^{\frac{n(4-n)}{2n}}\left(\frac{ r}{\sin r}\right)^{\frac{n^2-3n+2}{2n}},
\end{align}
and in order to keep the leftover power of $r$ positive we need $n \in \{3,4\}$.
}
\end{rmrk}
Now we can apply the inequality derived in Lemma \ref{LemmaCalculus} to obtain the following bound on $\phi'(r)$.

\bl\label{LemmaDerivativeSphericalMean}
Assume $u\in C^\infty(\Sph^n)$, $n \in \{3, 4\}$, $u\geq 0$ and $\Delta u\leq \frac{n(n-2)}{4}u$ then for $r\in \left(0,\frac{\pi}{2} \right)$ we have
\be
\phi'(r)
\leq 
\frac{n(n-2)}{4} \left( \frac{\pi}{2} \right)^{\frac{n+2}{2n}} \left( \omega_{n-1}\right)^{\frac{2-n}{2n}} \|u\|_{L^{\frac{2n}{n-2}}(\Sph^n)}.
\ee
\el

\begin{proof}
Following Lemma \ref{LemmaPhiDerivative}, we have 
\begin{align}
\phi'(r) 
&=\frac{\int_{B_r(x)} \Delta u dV_{g_{\Sph^n}}}{\omega_{n-1} \sin^{n-1} r}\\
&\leq \frac{n(n-2)}{4}\frac{\int_{B_r(x)} u dV_{g_{\Sph^n}}}{\omega_{n-1} \sin^{n-1}r} \\
&\leq \frac{n(n-2)}{4}\frac{\left(\int_{B_r(x)} u^{\frac{2n}{n-2}} dV_{g_{\Sph^n}}\right)^{\frac{n-2}{2n}}\Vol_{g_{\Sph^n}}(B_r(x))^{\frac{n+2}{2n}}}{\omega_{n-1} \sin^{n-1}r}\label{eq-HolderUsed} \\
& \leq \frac{n(n-2)}{4} \frac{\|u\|_{L^{\frac{2n}{n-2}}(\Sph^n)} \cdot \left| \omega_{n-1} \int^{r}_{0} \sin^{n-1}s ds\right|^{\frac{n+2}{2n}}}{\omega_{n-1} \sin^{n-1}r} \\
& \leq \frac{n(n-2)}{4} \left( \frac{\pi}{2} \right)^{\frac{n+2}{2n}}\left( \omega_{n-1}\right)^{\frac{2-n}{2n}} \|u\|_{L^{\frac{2n}{n-2}}(\Sph^n)}.
\end{align}
where we applied H\"{o}lder's inequality in \eqref{eq-HolderUsed} and the last step follows from Lemma \ref{LemmaCalculus}. 
\end{proof}

Next, by integrating the inequality in Lemma \ref{LemmaDerivativeSphericalMean} we obtain the following result.

\bl\label{lem: spherical mean inequality}
Assume $u\in C^\infty(\Sph^n)$, $n \in \{3, 4\}$, $u\geq 0$ and $\Delta u\leq \frac{n(n-2)}{4}u$. Then for any $x \in \Sph^n$, $0< r_0< r_1<\frac{\pi}{2}$, we have
\begin{align}
\begin{split}
&\dashint_{\partial B_{r_1}(x)} u d\sigma -\dashint_{\partial B_{r_0}(x)} u  d\sigma
\\&\leq  \frac{n(n-2)}{4} \left( \frac{\pi}{2} \right)^{\frac{n+2}{2n}}\left( \omega_{n-1}\right)^{\frac{2-n}{2n}} \|u\|_{L^{\frac{2n}{n-2}}(\Sph^n)} (r_1- r_0)
\end{split}
\end{align}
In particular, by letting $r_0 \to 0$, we have
\begin{equation}
u(x) \geq \dashint_{\partial B_r(x)}ud\sigma - \frac{n(n-2)}{4}\left( \frac{\pi}{2} \right)^{\frac{n+2}{2n}}\left( \omega_{n-1}\right)^{\frac{2-n}{2n}} \|u\|_{L^{\frac{2n}{n-2}}(\Sph^n)}r,
\end{equation}
for any $x \in \Sph^n$ and $0 < r < \frac{\pi}{2}$.
\el
\begin{proof}
By definition, we have
\[
\phi(r)=\dashint_{\partial B_{r}(x)} u d\sigma.
\]
By the fundamental theorem of calculus, we have
\be
\phi(r_1)-\phi(r_0) = \int_{r_0}^{r_1} \phi'(r)dr,
\ee
and so by applying Lemma \ref{LemmaDerivativeSphericalMean},  we obtain the desired conclusion. 
\end{proof}

If we have the an upper bound for $\|u\|_{L^\frac{2n}{n-2}(\Sph^n)}$, then we can conclude the following consequence of Lemma \ref{LemmaDerivativeSphericalMean}.
\bl\label{LemmaSphericalMeanInequality}
Assume $u\in C^\infty(\Sph^n)$, $n\in \{3, 4\}$, $u\geq 0$ and $\Delta u\leq \frac{n(n-2)}{4}u$. Assume
\be
\frac{n(n-2)}{4} \left( \frac{\pi}{2} \right)^{\frac{n+2}{2n}}\left( \omega_{n-1}\right)^{\frac{2-n}{2n}} 
\|u\|_{L^{\frac{2n}{n-2}}(\Sph^n)}
\leq K
\ee
for some $K>0$, then for any fixed $x \in \Sph^n$, and any $0< r_0< r_1<\frac{\pi}{2}$ we have
\be
\dashint_{\partial B_{r_1}(x)} u d\sigma -\dashint_{\partial B_{r_0}(x)} u d\sigma\leq  K(r_1- r_0)
\ee
or equivalently
\be\label{SphericalMeanInequality}
\dashint_{\partial B_{r_1}(x)} (u-Kr_1) d\sigma \leq \dashint_{\partial B_{r_0}(x)}( u-Kr_0) d\sigma.
\ee
\el
\begin{proof}
The proof follows from the fact that the function 
\begin{align}
   \dashint_{\partial B_{r}(x)} (u-Kr) d\sigma 
\end{align}
is a non-increasing function of $r$ for $r\in (0,\frac{\pi}{2})$.
\end{proof}

An interesting consequence of spherical mean inequality in Lemma \ref{lem: spherical mean inequality} is the following lemma which one should notice is similar to the Bishop-Gromov volume comparison lemma for Ricci curvature.

\bl \label{LemmaBishopGromov}
Assume $u\in C^\infty(\Sph^n)$, $n \in {3, 4} $, $u\geq 0$ and $\Delta u\leq \frac{n(n-2)}{4}u$. Assume
\be
\frac{n(n-2)}{4} \left( \frac{\pi}{2} \right)^{\frac{n+2}{2n}}\left( \omega_{n-1}\right)^{\frac{2-n}{2n}} 
\|u\|_{L^{\frac{2n}{n-2}}(\Sph^n)}
\leq K, 
\ee
for some $K>0$, then for any fixed $x \in \Sph^n$, and $0< r_0< r_1<\frac{\pi}{2}$ we have
\be
\dashint_{B_{r_1}(x)} (u-Kr) dV_{g_{\Sph^n}}  \leq \dashint_{B_{r_0}(x)}( u-Kr) dV_{g_{\Sph^n}}.
\ee
\el

\begin{proof}
We start by establishing the following inequality for all $r\in (0, \pi/2)$ we have 
\be
\dashint_{B_{r}(x)} (u(y)-K d_{\Sph^n}(y,x)) dV_{g_{\Sph^n}}(y)\geq \dashint_{\partial B_r(x)} (u-Kr)d\sigma.\label{Inequality1}
\ee
One can observe this by calculating
\begin{align}
 \int_{B_r(x)} (u(y)-K d_{\Sph^n}(y,x))  dV_{g_{\Sph^n}}
=&\int_0^r \left(\int_{\partial B_s(x)} (u-Ks) d\sigma \right)ds\\
=&\int_0^r (4\pi \sin^2 s)\left(\dashint_{\partial B_s(x)} (u-Ks) d\sigma \right)ds\quad \\
\geq & \int_0^r (4\pi \sin^2 s)\left(\dashint_{\partial B_r(x)} (u-Kr) d\sigma \right)ds\quad \label{eq-Reference1}\\
= &\vol_{g_{\Sph^n}}(B_r(x)) \dashint_{\partial B_r(x)} (u-Kr) d\sigma, 
\end{align}
where \eqref{eq-Reference1} follows by \eqref{SphericalMeanInequality} and $0<s\leq r$.

We now establish the next inequality for $0<r<R<\pi/2$, define $A_{r,R}(x)=B_R(x)\backslash B_r(x)$,  then we have
\be
\dashint_{A_{r,R}(x)}(u(y)-Kd_{\Sph^n}(y,x))dV_{g_{\Sph^n}}(y) \leq \dashint_{\partial B_r(x)} (u-Kr)d\sigma.\label{Inequality2}
\ee
We can see this inequality by calculating
\begin{align}
 &\int_{A_{r,R}(x)} (u(y) -K d_{\Sph^n}(y, x))  dV_{g_{\Sph^n}}(y) \\
=&\int_r^R \left(\int_{\partial B_s(x)} (u-Ks) d\sigma \right)ds\\
=&\int_r^R (4\pi \sin^2 r)\left(\dashint_{\partial B_s(x)} (u-Ks) d\sigma \right)ds\quad \\
\leq & \int_r^R(4\pi \sin^2 r)\left(\dashint_{\partial B_r(x)} (u-Kr) d\sigma \right)ds\quad \label{eq-Reference2}\\
= &\vol_{g_{\Sph^n}}(A_{r,R}(x)) \dashint_{\partial B_r(x)} (u-Kr) d\sigma,
\end{align}
where \eqref{eq-Reference2} follows by \eqref{SphericalMeanInequality} and $0<s\leq r$

Now by combining \eqref{Inequality1} with \eqref{Inequality2} we are able to establish the following inequality for $0< r_0< r_1<\frac{\pi}{2}$ 
\be
\dashint_{B_{r_1}(x)} (u-Kr) dV_{g_{\Sph^n}}  \leq \dashint_{B_{r_0}(x)}( u-Kr) dV_{g_{\Sph^n}}.\label{Inequality3}
\ee
Here we calculate
\begin{align}
& \int_{B_{r_1}(x)} (u -K d(y,x))dV_{g_{\Sph^n}} \\
=&\int_{B_{r_0}(x)} (u-K d(y,x))dV_{g_{\Sph^n}}
+\int_{A_{r_0,r_1}(x)} (u-K d(y,x))dV_{g_{\Sph^n}}\\
=&\vol_{g_{\Sph^n}}(B_{r_0}(x))\dashint_{B_{r_0}(x)} (u-K d(y,x))dV_{g_{\Sph^n}}
\\&\quad+\vol_{g_{\Sph^n}}(A_{r_0,r_1}(x))\dashint_{A_{r_0,r_1}(x)} (u-K d(y, x))dV_{g_{\Sph^n}}\\
\leq& \vol_{g_{\Sph^n}}(B_{r_0}(x))\dashint_{B_{r_0}(x)} (u-K d(y,x))dV_{g_{\Sph^n}}
\\&\quad +\vol_{g_{\Sph^n}}(A_{r_0,r_1}(x))\dashint_{B_{r_0}(x)} (u-K d(y, x))dV_{g_{\Sph^n}}\\
=&  \vol_{g_{\Sph^n}}(B_{r_1}(x))\dashint_{B_{r_0}(x)} (u-K d(y, x))dV_{g_{\Sph^n}}.
\end{align}
This finishes the proof.
\end{proof}

We can use Lemma \ref{LemmaBishopGromov} to prove that the limit conformal factor is defined everywhere and that it will be lower semicontinuous.

\bp\label{prop: lower semicontinuous u}
Assume $u_i \in C^\infty(\Sph^n)$, $n\in \{3, 4\}$, $u_i \geq 0$ and $\Delta u_i \leq \frac{n(n-2)}{4}u_i$. Assume
\be
\frac{n(n-2)}{4} \left( \frac{\pi}{2} \right)^{\frac{n+2}{2n}}\left( \omega_{n-1}\right)^{\frac{2-n}{2n}} 
\| u_i \|_{L^{\frac{2n}{n-2}}(\Sph^n)}
\leq K
\ee
for some $K>0$. Also assume that 
\begin{align}
    \int_{\Sph^n}|u_i-u_\infty|dV_{g_{\Sph^n}}\to 0 \text{ as } i\to \infty,
\end{align}  
for some $u_\infty\in L^1(\Sph^n)$, then $u_\infty: \Sph^n \to (0,\infty]$ is defined everywhere in $\Sph^n$ and $u_\infty$ is lower semicontinuous.
\ep

\begin{proof}
Apply Lemma \ref{LemmaBishopGromov} to each $u_i$ then we get
\be\label{BGi}
\dashint_{B_{r_1}(x)} (u_i-Kr) dV_{g_{\Sph^n}}  \leq \dashint_{B_{r_0}(x)}( u_i-Kr) dV_{g_{\Sph^n}}
\ee
for any $i$ and for any $0< r_0< r_1<\frac{\pi}{2}$. Since $u_i$ converges to $u_\infty$ in $L^1(\Sph^n)$ norm, we also get
\be
\dashint_{B_{r_1}(x)} (u_\infty-Kr) dV_{g_{\Sph^n}}  \leq \dashint_{B_{r_0}(x)}( u_\infty-Kr) dV_{g_{\Sph^n}}
\ee
for any $0< r_0< r_1<\frac{\pi}{2}$.

As a result, we can define
\be\label{BGinf}
u_\infty (x)=\lim_{s\to 0}\dashint_{B_{s}(x)} (u_\infty-Kr) dV_{g_{\Sph^n}}.
\ee
In addition, we can think of it as
\be
u_\infty (x)=\sup_{s\in (0, \frac{\pi}{2})}\dashint_{B_{s}(x)} (u_\infty-Kr) dV_{g_{\Sph^n}},
\ee
where the function $\dashint_{B_{s}(x)} (u_\infty-Kr) dV_{g_{\Sph^n}} $ is continuous with respect to both $x$ and $s$. Since the supremum of lower semicontinuous functions is lower semicontinuous, we conclude that $u_\infty$ is lower semicontinuous in $\Sph^n$.

\end{proof}


\subsection{Spherical mean inequality for $u^{\frac{2}{n-2}}$ on the $n$-sphere ($n \geq 5$)}\label{subsect: spherical mean inequality high dimensions}

Now we apply the spherical mean method to $u^{\frac{2}{n-2}}$ which will yield interesting results in dimensions $n \geq 5$.

\bl\label{Lem-Superharmonic}
Let $u: \Sph^n \to \R$ be a positive smooth function. Let $C>0$ be a real number such that
\begin{equation}
\Delta u \leq Cu.
\end{equation}
Then for any $\alpha\in (0,1)$, we have
\begin{equation}
\Delta u^\alpha \leq \alpha Cu^\alpha.
\end{equation}

In particular, 
\begin{equation}
\Delta u^{\frac{2}{n-2}} \leq \frac{2}{n-2} C u^{\frac{2}{n-2}}
\end{equation}
for $ n \geq 5$.
\el
\begin{proof}
By direct calculation,
\begin{equation}
\Delta  u^\alpha =  \alpha u^{\alpha-1}\Delta  u+ \alpha(\alpha-1) |\nabla u|^2 u^{\alpha-2}.
\end{equation}
Since $\alpha\in (0,1)$, $\alpha-1<0$. Hence
\begin{equation}
\Delta  u^\alpha \leq   \alpha u^{\alpha-1}\Delta  u.
\end{equation}
Apply the inequality \(\Delta u \leq Cu,\) then we have
\begin{equation}
\Delta u^\alpha \leq \alpha Cu^\alpha.
\end{equation}
\end{proof}

Another elementary lemma we need is following:
\begin{lem}\label{lem: elementary inequality}
For all $ r \in \left( 0, \frac{\pi}{2} \right)$, 
\begin{equation}
\frac{\left( \int^r_0 \sin^{n-1} s  ds\right)^{\frac{n-1}{n}}}{\sin^{n-1} r} \leq \left( \frac{\pi}{2} \right)^{\frac{n-1}{n}}
\end{equation}
holds.
\end{lem}
\begin{proof}
By using integral mean equality, for any $ r \in \left(0, \frac{\pi}{2} \right)$, there exists $\xi \in (0, r)$ such that 
\begin{eqnarray}
\frac{\left( \int^r_0 \sin^{n-1} s  ds\right)^{\frac{n-1}{n}}}{\sin^{n-1} r}
&= & \frac{\left( r \sin^{n-1} \xi \right)^{\frac{n-1}{n}}}{\sin^{n-1} r}  \\
& \leq & \frac{\left( r \sin^{n-1} r \right)^{\frac{n-1}{n}}}{\sin^{n-1} r} 
= \left( \frac{\sin r}{r} \right)^{\frac{n-1}{n}}
\leq \left( \frac{\pi}{2} \right)^{\frac{n-1}{n}}.
\end{eqnarray}
\end{proof}

\bl\label{lem: spherical mean inequality high dim}
Assume $u\in C^\infty(\Sph^n)$, $n \geq 5$, $u\geq 0$ and $\Delta u\leq \frac{n(n-2)}{4}u$. Then for any $x \in \Sph^n$, $0< r_0< r_1<\frac{\pi}{2}$, we have
\begin{equation}
\dashint_{\partial B_{r_1}(x)} u d\sigma -\dashint_{\partial B_{r_0}(x)} u  d\sigma
\leq  \frac{n}{2} \left( \frac{\pi}{2} \right)^{n-1} \frac{1}{\left( \omega_{n-1} \right)^{\frac{1}{n}}}\|u\|^{\frac{2}{n-2}}_{L^{\frac{2n}{n-2}}(\Sph^n)} (r_1- r_0)
\end{equation}
In particular, by letting $r_0 \to 0$, we have
\begin{equation}
u(x) \geq \dashint_{\partial B_r(x)}ud\sigma - \frac{n}{2} \left( \frac{\pi}{2} \right)^{n-1} \frac{1}{\left( \omega_{n-1} \right)^{\frac{1}{n}}}\|u\|^{\frac{2}{n-2}}_{L^{\frac{2n}{n-2}}(\Sph^n)} r,
\end{equation}
for any $x \in \Sph^n$ and $0 < r < \frac{\pi}{2}$.
\el
\begin{proof}
First of all, by Lemma \ref{Lem-Superharmonic}, we have
\begin{equation}
\Delta u^{\frac{2}{n-2}}
 \leq
\frac{n}{2} u^{\frac{2}{n-2}}.
\end{equation}
Then by Lemma \ref{LemmaPhiDerivative}, we have that for any $0 < r < \frac{\pi}{2}$,
\begin{eqnarray}
\frac{d}{dr} \left( \dashint_{\partial B_{r}(x)} u^{\frac{2}{n-2}} d\sigma \right)
& = &
\frac{\int_{B_r(x)} \Delta u^{\frac{2}{n-2}} dV_{g_{\Sph^n}}}{\omega_{n-1}  \sin^{n-1} r} \\
& \leq & 
\frac{n}{2} \frac{\int_{B_r(x)} u^{\frac{2}{n-2}} dV_{g_{\Sph^n}}}{\omega_{n-1} \sin^{n-1}r} \\
& \leq & 
\frac{n}{2} \frac{\|u\|^{\frac{2}{n-2}}_{L^{\frac{2n}{n-2}}} \left( \int^r_0 \omega_{n-1} \sin^{n-1} s ds\right)^{\frac{n-1}{n}}}{ \omega_{n-1} \sin^{n-1} r} \\
& \leq & 
\frac{n}{2} \left( \frac{\pi}{2} \right)^{n-1} \frac{1}{\left( \omega_{n-1} \right)^{\frac{1}{n}}} \|u\|^{\frac{2}{n-2}}_{L^{\frac{2n}{n-2}}(\Sph^n)}
\end{eqnarray}
In the last step, we use Lemma \ref{lem: elementary inequality}. Then integrating the inequality provides the conclusion and completes the proof.
\end{proof}

In the same way as the proof of Lemma \ref{LemmaBishopGromov}, by applying Lemma \ref{lem: spherical mean inequality high dim}, we obtain the following ball average inequality of $u^{\frac{2}{n-2}}$, which will be used in the proof of lower semi-continuity of $u$ in Proposition \ref{prop: lower semi-continuity of u high dim}.
\begin{lem}\label{lem: ball average inequality high dim}
Assume $u \in C^\infty(\Sph^n)$, $n \geq 5$, $u \geq 0$ and $\Delta u \leq \frac{n(n-2)}{4}u$. Assume
\begin{equation}
\frac{n}{2} \left( \frac{\pi}{2} \right)^{n-1} \frac{1}{\left( \omega_{n-1} \right)^{\frac{1}{n}}} \|u\|^{\frac{2}{n-2}}_{L^{\frac{2n}{n-2}}(\Sph^n)} \leq K,
\end{equation}
for some $K > 0$, then for any fixed $ x\in \Sph^n$, and $0 < r_0 < r_1 < \frac{\pi}{2}$ we have
\begin{equation}
\dashint_{B_{r_1}(x)} \left( u^{\frac{2}{n-2}} - K r \right) dV_{g_{\Sph^n}} \leq \dashint_{B_{r_{0}}(x)} \left( u^{\frac{2}{n-2}} - K r \right) dV_{g_{\Sph^n}}.
\end{equation}
\end{lem}


\subsection{Spherical mean inequality of $f = \frac{n-2}{2} \ln u$ on $\Sph^n$ ($n \geq 3$)}\label{subsect: spherical mean inequality for f}

Now we use the same method as in Subsection \ref{subsect: spherical mean lower dim} and Subsection \ref{subsect: spherical mean inequality high dimensions} to show similar results for $f_j$.

\bl\label{lem: spherical mean inequality of f}
Assume that $f\in C^\infty(\Sph^n)$, $ n \geq 3$, satisfies $\Delta f\leq \frac{n}{2}$. Then for any fixed $x\in \Sph^n$ and any $0< r_0< r_1< \frac{\pi}{2}$,
\begin{equation}
\dashint_{\partial B_{r_1}(x)} f d\sigma -\dashint_{\partial B_{r_0}(x)} f  d\sigma
\leq  C(n) (r_1 - r_0),
\end{equation}
or equivalently
\begin{equation}
\dashint_{\partial B_{r_1}(x)} \left( f - C(n)r_1 \right) d\sigma
\leq
\dashint_{\partial B_{r_0}(x)} \left( f- C(n)r_0 \right) d\sigma,
\end{equation}
where $C(n):= 
\frac{n}{2} \int^{\frac{\pi}{2}}_0 \sin^{n-1} rdr$.
\el 
\begin{proof}
By Lemma \ref{LemmaPhiDerivative}, and $\Delta f \leq \frac{n}{2}$, we have that for any $0 < r < \frac{\pi}{2}$,
\begin{eqnarray}
\frac{d}{dr} \left( \dashint_{\partial B_{r}(x)} f d\sigma \right)
& = &
\frac{\int_{B_r(x)} \Delta f dV_{g_{\Sph^n}}}{\omega_{n-1} \sin^{n-1} r} \\
& \leq & 
\frac{n}{2} \frac{\int^r_0 \sin^{n-1}s ds}{\sin^{n-1}r} \\
& \leq & 
\frac{n}{2} \frac{\int^{\frac{\pi}{2}}_0 \sin^{n-1}s ds}{\sin^{n-1}\frac{\pi}{2}} =: C(n).
\end{eqnarray}
In the last step, we use a basic fact that $\frac{n}{2} \frac{\int^r_0 \sin^{n-1}s ds}{\sin^{n-1}r}$ is a increasing function of $r$. Then integrating the inequality provides the conclusion and completes the proof.
\end{proof}

In the same way as the proof of Lemma \ref{LemmaBishopGromov}, by applying Lemma \ref{lem: spherical mean inequality of f}, we obtain the following ball average inequality of $f$.
\begin{lem}\label{lem: BishopGromov of f}
Assume that $f\in C^\infty(\Sph^n)$, $n \geq 3$, satisfies $\Delta f\leq \frac{n}{2}$. Then for any fixed $x \in \Sph^n$ and any $0< r_0< r_1< \frac{\pi}{2}$,
\begin{equation}
\dashint_{B_{r_1}(x)}\left( f - C(n)r \right) dV_{g_{\Sph^n}}
\leq
\dashint_{B_{r_0}(x)}\left( f- C(n)r \right) dV_{g_{\Sph^n}}
\end{equation}
where 
$C(n) = 
\frac{n}{2} \int^{\frac{\pi}{2}}_0 \sin^{n-1} rdr$.

In particular, by taking limit as $r_0 \to 0$, one obtains that
\begin{equation}
\dashint_{B_{r_1}(x)} \left( f - C(n) r \right) dV_{g_{\Sph^n}} 
\leq f(x)
\end{equation}
holds for any $x \in \Sph^n$ and any $0 < r_1 < \frac{\pi}{2}$.
\end{lem}

Now by the same argument as in the proof of Proposition \ref{prop: lower semicontinuous u}, the ball average inequality in Lemma \ref{lem: BishopGromov of f} implies the following:
\begin{prop}\label{prop: lower semicontinuous f}
Assume that the sequence of functions $f_i \in C^\infty(\Sph^n)$, $n \geq 3$, satisfies $\Delta f \leq \frac{n}{2}$ for all $i \in \N$, and $f_i \to f_\infty$ in $L^1(\Sph^n)$ as $i \to \infty$. Then $f_\infty: \Sph^n \to (-\infty, +\infty]$ is defined everywhere in $\Sph^n$ and $f_\infty$ is lower semicontinuous.
\end{prop}



\section{Uniform positive lower bound of conformal factors}\label{sec: Uniform Lower Bound}

In this section our main goal is to show that the limiting conformal factor must be uniformly bounded away from zero in Proposition \ref{prop: positive lower bound}. This provides a non-degeneracy result for the geometry defined by the limiting conformal factor. As an application of Proposition \ref{prop: positive lower bound}, we obtain a positive uniformly lower bound of the sequence of conformal factors $u_i$ in Proposition \ref{prop: uniform positive lower bound}. By applying Proposition \ref{prop: uniform positive lower bound}, we further prove a convergence result for $f_i = \frac{2}{n-2} \ln u_i$ in Corollary \ref{cor: uniform L^2 bound on f_j}, which will be used in Section \ref{sec: Singular Set Decomposition}. Another important application of Proposition \ref{prop: positive lower bound} is to obtain $L^1$ convergence of $u^{\frac{2}{n-2}}_i$ in dimensions of $n \geq 5$ in Lemma \ref{lem: L1 convergence of a power of u}, which is used in the proof of Proposition \ref{prop: uniform positive lower bound}, as well as in the proof of lower semi-continuity of the limiting conformal factor $u_\infty$ in Proposition \ref{prop: lower semi-continuity of u high dim}.

\subsection{Positive infimum of limiting conformal factor}\label{subsect: positive lower bound}

We start by establishing a dichotomy property for the limit conformal factor as in Proposition \ref{prop-cutoff-infimum}. To this end, we begin by proving some useful lemmas.
\bl\label{Lem-Log-L2}
Let $n\geq 3$ be a positive integer. Let $u: \Sph^n \to \R$ be a positive smooth function. Let $C>0$ be a real number such that
\be
\Delta u \leq Cu.
\ee
If we define $f=  \ln u$, then
\be
\| \nabla f\|_{L^2 (\Sph^n)}\leq\sqrt{C\vol (\Sph^n)}.
\ee
\el
\begin{proof}
Since $u= e^{ f}$, we have
\[\Delta  u=  e^{ f}\Delta f+ |\nabla f|^2 e^{ f},\]
as a result, the inequality $\Delta u \leq Cu$ is equivalent to 
\be
 \Delta f+ |\nabla f|^2 \leq C.
\ee
Integrate over $\Sph^n$ and use Stokes' Theorem then we have the desired result.
\end{proof}

We now define the truncation of a function which will be used in the results which follow.
\bd\label{defn-cut-off}
Let $n\geq 3$ be a positive integer. Let $u: \Sph^n \to \R$ be a positive smooth function. Let $K>0$ be a real number,  for each $x \in \Sph^n$, we define
\be
\bar{u}^K(x)=
\begin{cases}
u(x), & \text{ if }  \ \ u(x)<K,\\
K, & \text{ if } \ \ u(x)\geq K.
\end{cases}
\ee
Then $\bar{u}^K$ is a positive continuous function on $\Sph^n$ with the maximal value not greater than $K$.
\ed
From the definition we can prove the following lemma:
\bl\label{Lemma-Cutoff-Single-Function}
Let $n\geq 3$ be a positive integer. Let $u: \Sph^n \to \R$ be a positive smooth function, and let $K>0$ be a regular value of the function $u$. If for some real number $C>0$ we have
\be
\Delta u\leq C u \text{ \rm in }\Sph^n
\ee
then for all $\varphi \in W^{1,2}(\Sph^n)$ such that $\varphi\geq 0$ we have
\be
-\int_{\Sph^n} \langle \nabla \varphi, \nabla \bar{u}^K \rangle dV_{g_{\Sph^n} }\leq \int_{\Sph^n} \varphi \bar{u}^KdV_{g_{\Sph^n} }.
\ee
\el

\begin{proof}
By Theorem 4.4 from \cite{Evans-Gariepy}, we have for all $K>0$
\be
\nabla \bar{u}^K=\begin{cases}
\nabla u,  & \text{a.e. on } \{u(x)<K\},\\
0, & \text{a.e. on } \{u (x)\geq K\}.
\end{cases}
\ee
As a result we have
\be
\begin{split}
-\int_{\Sph^n} \langle \nabla \varphi, \nabla \bar{u}^K \rangle dV_{g_{\Sph^n} } &=-\int_{\{ u<K\}} \langle \nabla \varphi, \nabla u \rangle dV_{g_{\Sph^n} }\\
&=\int_{\{u <K\}} \varphi \Delta udV_{g_{\Sph^n} }-\int_{\partial\{u<K\}} \varphi \partial_{\nu} udA_{g_{\Sph^n} },
\end{split}
\ee
where $dA_{g_{\Sph^n} }$ is the area form induced on $\partial\{u<K\}$ which we know is well defined by the next observation. Since $K$ is a regular value of $f$, from the regular level set Theorem  we know that the level set $ \{u=K\}=\partial \{u<K\}$ is an embedded submanifold of dimension $n-1$ in $\Sph^n$. Hence we can apply Stokes' theorem to get the last step. Moreover, since $\nu $ is the outer unit normal vector on the boundary of the set $ \{u<K\}$, we have
\be
 \partial_{\nu} u \geq 0 \text{\rm\  for all } x\in \partial \{u<K\}.
\ee
Hence we can drop the boundary term to get the inequality
\be
-\int_{\Sph^n} \langle \nabla \varphi , \nabla \bar{u}^K \rangle dV_{g_{\Sph^n} }\leq \int_{\{u <K\}} \varphi \Delta u dV_{g_{\Sph^n} }.
\ee
Since
\be
 \Delta u \leq C u,
\ee
we have
\begin{align}
-\int_{\Sph^n} \langle \nabla \varphi, \nabla \bar{u}^K \rangle dV_{g_{\Sph^n} }&\leq  \int_{\{u <K\}} \varphi \Delta u dV_{g_{\Sph^n} } 
\\&\leq  C\int_{\{u <K\}} \varphi u dV_{g_{\Sph^n} }\leq  C \int_{\Sph^n} \varphi \bar{u }^K dV_{g_{\Sph^n} }.
\end{align}
This finishes the proof.
\end{proof}

We can prove similar results for a sequence of functions:
\bl\label{Cutoff-Lemma-Sard}
Let $n\geq 3$ be a positive integer. Let $u_j$ be a sequence of smooth positive functions defined on $\Sph^n$. If for some $C>0$ we have
\be
\Delta u_j \leq C u_j, \ \ \forall j \in \N,
\ee
then there exists $K>0$ such that for all $\varphi \in W^{1,2}(\Sph^n)$ with $\varphi\geq 0$ we have
\be\label{Cutoff-Inequality}
-\int_{\Sph^n} \langle \nabla \varphi,\nabla \bar{u}_j^K \rangle dV_{g_{\Sph^n} }\leq  C\int_{\Sph^n} \varphi\bar{u}_j^K dV_{g_{\Sph^n} } \quad \forall j \in \N .
\ee
Moreover, we can choose $K$ as large as we want.
\el

\begin{proof}
Note that if $0<K \leq \inf\limits_{x\in \Sph^n} u_j(x)$ for some $j$ then we have $\bar{u}_j^K(x)=K$ for all $x\in \Sph^n$. On the other hand, if $\sup\limits_{x\in\Sph^n}u_j(x) \leq K$ for some $j$ then $\bar{u}_j^K(x)=u_j(x)$ for all $x\in \Sph^n$. In either one of these two cases, the inequality (\ref{Cutoff-Inequality}) holds.

In general, by Sard's theorem, for each function $u_j$, the critical values of $u_j$ has measure zero, and the union of all the critical sets for each of the function also has measure zero. As a result, there exists $K>0$ such that for each $u_j$ either $K$ is a regular value or $u_j^{-1}(\{K\})=\emptyset$. By Lemma \ref{Lemma-Cutoff-Single-Function} we get inequality (\ref{Cutoff-Inequality}). Moreover, we can choose $K$ as large as we want. This finishes the proof.
\end{proof}

Next we prove similar results for the limit function, but before that we need to consider the regularity of the limit function:
\bl\label{Cutoff-Lemma-Limit}
Let $n\geq 3$ be a positive integer. Let $K$ and $C$ be positive real numbers. Let $u_{j}$ be a sequence of positive smooth functions on $\Sph^n$ satisfying
\be
\Delta u_{j} \leq C u_j, \quad \forall j \in \N.
\ee
Then the sequence $\bar{u}^K_j$ is uniformly bounded in $W^{1,2}(\Sph^n)$:
\be
\|\bar{u}^K_j \|_{W^{1,2}(\Sph^n)}\leq K\sqrt{(1+C) \vol(\Sph^n)}.
\ee
As a result, there exists $\bar{u}_{\infty}^K\in W^{1,2}(\Sph^n)$ such that $\bar{u}^K_j $ converges to $\bar{u}^K_\infty $ in $L^2(\Sph^n)$, and that $\bar{u}^K_j $ converges to $\bar{u}^K_\infty $ weakly in $W^{1,2}(\Sph^n)$.
\el

\begin{proof}
By definition of the truncation in Definition \ref{defn-cut-off}, we get
\be\label{Cutoff-L2}
\|\bar{u}_j^K\|_{L^{2}(\Sph^n)}\leq K \sqrt{\vol(\Sph^n)}.
\ee
By Theorem 4.4 from \cite{Evans-Gariepy}, we have for all $K>0$ and for each $j$
\be
\nabla \bar{u}^K_j=\begin{cases}
\nabla u_j, & \text{a.e. on } \{u_j (x)<K\},\\
0,  & \text{a.e. on } \{u_j (x)\geq K\}.
\end{cases}
\ee
Hence
\be\label{Cutoff-Gradient}
\begin{split}
\|\nabla \bar{u}^K_j\|^2_{L^2(\Sph^n)} &=\int_{\{f_j <K\}} |\nabla u_j|^2dV_{g_{\Sph^n} } \\
&=\int_{\{u_j <K\}} |u_j |^2|\nabla \ln u_j |^2dV_{g_{\Sph^n} }\\
&\leq K^2\int_{\{u_j <K\}}  |\nabla \ln u_j |^2dV_{g_{\Sph^n} }\\
&\leq K^2\|\nabla \ln u_j\|^2_{L^2(\Sph^n)}\leq K^2 C \vol(\Sph^n),
\end{split}
\ee
where the last step follows from Lemma \ref{Lem-Log-L2}. Combine inequalities (\ref{Cutoff-L2}) and (\ref{Cutoff-Gradient}) then we get the desired results.
\end{proof}

Now we prove the following proposition concerning the limit function:
\begin{lem}\label{lem-differential inequality for limit function}
Let $u_{j}$ be a sequence of positive smooth functions on $\Sph^n$, $n \geq 3$, satisfying
\be
\Delta u_{j} \leq u_j, \quad \forall j \in \N.
\ee
Let $K>0$ be a real number that satisfies the requirement in Lemma \ref{Cutoff-Lemma-Sard}.
Let $\bar{u}^K_{\infty} \in W^{1, 2}(\Sph^n)$ be the limit function as in Lemma \ref{Cutoff-Lemma-Limit}. Then $\bar{u}^K_{\infty}$ satisfies the inequality
\be
-\int_{\Sph^n} \langle \nabla \varphi,\nabla \bar{u}_{\infty}^K \rangle dV_{g_{\Sph^n} }\leq \int_{\Sph^n} \varphi \bar{u}_{\infty}^KdV_{g_{\Sph^n} },
\ee
for all $\varphi \in W^{1,2}(\Sph^n)$ such that $\varphi\geq 0$.
\end{lem}
\begin{proof}
By Lemma \ref{Cutoff-Lemma-Limit} we know that $\bar{u}^K_j $ converges to $\bar{u}^K_\infty $ in $L^2(\Sph^n)$, and that $\bar{u}^K_j $ converges to $\bar{u}^K_\infty $ weakly in $W^{1,2}(\Sph^n)$. As a result, for any $\varphi \in W^{1,2}(\Sph^n)$ we have that
\be
\int_{\Sph^n} \varphi\bar{u}_{j}^K dV_{g_{\Sph^n} }\to \int_{\Sph^n} \varphi \bar{u}_{\infty}^KdV_{g_{\Sph^n} },  \ \ \text{ as }j\to \infty,
\ee
and that
\be
\int_{\Sph^n} \langle \nabla \varphi ,\nabla \bar{u}_{j}^K \rangle dV_{g_{\Sph^n} } \to \int_{\Sph^n} \langle \nabla \varphi,\nabla \bar{u}_{\infty}^K \rangle dV_{g_{\Sph^n} }, \ \  \text{ as } j\to\infty.
\ee
As a result, by (\ref{Cutoff-Inequality}) we have for all $\varphi\in W^{1,2}(\Sph^n)$ such that $\varphi \geq 0$
\be
- \int_{\Sph^n} \langle \nabla \varphi,\nabla \bar{u}_{\infty}^K\rangle dV_{g_{\Sph^n} } \leq \int_{\Sph^n} \varphi\bar{u}_{\infty}^K dV_{g_{\Sph^n} }.
\ee
This finishes the proof.
\end{proof}

We need the definition of essential infimum of a function:
\begin{defn}\label{Defn-Ess-Inf}
Consider the standard round sphere $(\Sph^n,g_{\Sph^n})$, $ n\geq 3$. Let $U$ be an open subset of $\Sph^n$ . Let $f: U \to \R$ be measurable. Define the set
\be
U_f^{ess}=\{a\in \R: Vol_{g_{\Sph^n}}(f^{-1}(-\infty,a))=0\}.
\ee
We use $\inf\limits_{U} f$ to denote the essential infimum of $f$ in $U$ and define
\be
\inf_{U} f=\sup U^{ess}_f
\ee
\end{defn}

Finally, we apply the maximum principle for weak solution to prove the following property for the essential infimum of  $u_\infty$.

\begin{prop}\label{prop-cutoff-infimum}
Let $u_{j}$ be a sequence of positive smooth functions on $\Sph^n$, $ n \geq 3$, satisfying
\be
\Delta u_{j} \leq C u_j, \quad \forall j \in \N.
\ee
If we further assume that $u_j\to u_\infty$ in $L^2(\Sph^n)$ for some $u_\infty$, then either the essential infimum of $u_\infty$ is strictly positive or $u_\infty = 0$ a.e. on $\Sph^n$.
\end{prop}
\begin{proof}
Since $\|u_j-u_\infty\|_{L^2(\Sph^n)}\to 0$ as $j\to\infty$, choose a subsequence if needed, then we have $u_j\to u_\infty$ pointwise almost everywhere in $\Sph^n$. Let $K>0$ be a real number that satisfies the requirement in Lemma \ref{Cutoff-Lemma-Sard}. Construct a truncated sequence $\bar{u}^K_j$ as in Definition \ref{defn-cut-off}.
By Lemma \ref{Cutoff-Lemma-Limit}, choose a subsequence if needed, there exists $\bar{u}^K_{\infty}\in W^{1,2}(\Sph^n)$ such that $\bar{u}^K_j$ converges to $\bar{u}^K_\infty$ in $L^2(\Sph^n)$ norm. As a result, choose a subsequence if needed we have $\bar{u}_j^K\to \bar{u}_\infty^K$ pointwise almost everywhere in $\Sph^n$.

It suffices to show that if the essential infimum $\inf\limits_{\Sph^n} u_\infty = 0$ then $\bar{u}^K_\infty  = u_\infty = 0$ in $\Sph^n$. We assume that $\inf\limits_{\Sph^n} u_\infty = 0$. Since for each $j$ we have $0<\bar{u}_j^K\leq u_j$, we have $0\leq \inf\limits_{\Sph^n} \bar{u}_j^K \leq \inf\limits_{\Sph^2} u_\infty =0$.
This implies that for any $\delta, \delta^{\prime}>0$, we have
\be
\Vol_{g_0}\left(\left(\bar{u}^K_\infty\right)^{-1}(-\infty, \delta)\right) >0,
\ee
and
\be
\Vol_{g_0}\left(\left(\bar{u}^K_\infty \right)^{-1}(-\infty, -\delta^{\prime})\right)=0.
\ee
Let $N$ be the north pole of $\mathbb{S}^{n}$, and $S$ be the south pole.  $B_{\frac{\pi}{2}}(N)$ and $B_{\frac{\pi}{2}}(S)$ are upper and lower hemispheres respectively. Then either
\be
\inf\limits_{B_{\frac{\pi}{2}}(N)} \bar{u}^K_\infty=0,
\ee
or
\be
\inf\limits_{B_{\frac{\pi}{2}}(S)} \bar{u}^K_\infty=0.
\ee
Without loss of generality we assume that $\inf\limits_{B_{\frac{\pi}{2}}(N)}\bar{u}^K_\infty=0$. Since $\bar{u}^K_\infty\geq 0$ in $\Sph^n$, for any $r>\frac{\pi}{2}$, and $\epsilon>0$ such that $r+\epsilon <\pi$ we have
\be\label{eqn: inf-equal}
\inf_{B_{r}(N)} \bar{u}^K_\infty = \inf_{B_{r+\epsilon}(N)}\bar{u}^K_\infty=0.
\ee

Now by Lemma \ref{lem-differential inequality for limit function}, $\bar{u}^K_\infty$ satisfies
\be
(\Delta -C)\bar{u}^K_\infty \leq 0,
\ee
on $B_{r+\epsilon}(N)$ in the weak sense.
Hence by the strong maximum principle for weak solutions (see Theorem 8.19 in \cite{Gilbarg-Trudinger}), the equality in (\ref{eqn: inf-equal}) implies that $\bar{u}^K_{\infty}$ is constant on $B_{r}(N)$. This is true for any $r>\frac{\pi}{2}$, thus $\bar{u}^K_\infty \equiv0$ on $\mathbb{S}^{2}$. Moreover, since $K>0$, for almost every $x\in \Sph^n$ we have,
\be
\lim_{j\to\infty} \bar{u}_j^K=\lim_{j\to\infty} u_j=0,
\ee
and hence $u_\infty = 0$ a.e. on $\Sph^n$. This finishes the proof.
\end{proof}

We now assume that the volume of the sequence is bounded below and that the sequence is uniformly integrable; from these assumptions, we obtain that $u_{\infty}\neq 0$ in $L^p(\Sph^n)$.
\bl\label{lem-uniform-integrability}
Let $u_i$ be a sequence of smooth function in $\Sph^n$, $n\geq 3$, such that $u_i >0$. Assume that the sequence of metrics $ g_i := u_i^{\frac{4}{n-2}} g_{\Sph^n}$ has uniform volume lower bound such that
\be
V\leq \vol_{g_i}(\Sph^n) \text{ for all }i
\ee
for some $V>0$. We also assume that the sequence $u_i$ satisfies the uniform integrability condition, in the sense that for some $\alpha \in (0,1)$, $\Lambda >0$ we have
\be
\vol_{g_i}(U) \le  \Lambda \vol_{g_{\Sph^n}}(U)^{\alpha}, \quad \forall U \subset M \text{ measurable}.
\ee
If we further assume that for some $p\in [1,\infty)$ and some  $u_\infty \in L^p(\Sph^n)$, we have $\|u_i-u_\infty\|_{L^p(\Sph^n)}\to 0$, then 
\be
u_\infty\not \equiv 0 \ \ \text{in} \ \ L^p(\Sph^n).
\ee
\el

\begin{proof}
By contradiction, we assume that
\be
u_\infty\equiv 0.
\ee
Since $u_i$ converges to $u_\infty$ in $L^p(\Sph^n)$, by Theorem 4.9 in \cite{Brezis} we have $u_i\to 0$ pointwise almost everywhere.

Choose $\epsilon>0$ such that $\Lambda \epsilon^\alpha\leq V/2$, then by Egorov's theorem (Theorem 4.29 in \cite{Brezis}), there exists $A\subset \Sph^n$ measurable such that $\vol_{g_{\Sph^n}}(\Sph^n\backslash A) <\epsilon$, and that after passing to a subsequence (which we still denote as $u_i$), $u_i\to 0$ uniformly in $A$.

By the uniform integrability assumption, we have
\be
\vol_{g_i}(\Sph^n\backslash A) \leq\Lambda (\vol_{g_{\Sph^n}}(\Sph^n\backslash A))^\alpha< \Lambda \epsilon^\alpha < V/2.
\ee

Since $u_i\to  0$ uniformly in $A$, we can also choose $i$ large enough, such that
\be
\vol_{g_i}(A) =\int_{A} u_i^{\frac{2n}{n-2}} dV_{g_{\Sph^n}} <V/2.
\ee

Combine these results, then we have when $i$ is large enough,

\be
\vol_{g_i}(\Sph^n)=\vol_{g_i}(A)+\vol_{g_i}(\Sph^n\backslash A)< V,
\ee
which is a contradiction.
\end{proof}

Now we are ready to prove that the limit function $u_\infty := e^{\frac{(n-2)f_\infty}{2}}$ obtained in Corollary \ref{cor:W12WeakConvergence} has positive essential infimum.

\begin{prop}\label{prop: positive lower bound}
Let $u_i$ be a sequence of smooth positive functions in $\Sph^n$, $n \geq 3$. Assume that the metrics $ g_i := u_i^{\frac{4}{n-2}} g_{\Sph^n}$ have nonnegative scalar curvature, i.e. $\Sc_{g_i} \geq 0$, and uniformly bounded volumes, i.e.
\be
V^{-1} \leq \vol_{g_i}(\Sph^n) \leq V \text{ for all } \, i \in \N,
\ee
for some $V>0$. We also assume that the sequence $u_i$ satisfies the uniform integrability condition, in the sense that for some $\alpha \in (0,1)$, $\Lambda>0$
\be
\vol_{g_i}(U) \le  \Lambda \vol_{g_{\Sph^n}}(U)^{\alpha}, \quad \forall U \subset M \text{ measurable}.
\ee
Then the limit function $u_\infty := e^{\frac{(n-2)f_\infty}{2}} \in W^{1, p}$, with $1 \leq p < \frac{4n}{3n-2}$, has positive essential infimum, i.e. there exists $e_\infty >0$ such that $\inf\limits_{\Sph^n} u_\infty \geq e_\infty$. Here the limit function is obtained in Corollary \ref{cor:W12WeakConvergence}.
\end{prop}

\begin{proof}
By Corollary \ref{cor:W12WeakConvergence} with $f_i := \frac{2 \ln u_i}{n-2}$, $\Sc_{g_i} \geq 0$ and $\vol_{g_i}(\Sph^n) \geq V^{-1}$ for all $i \in \N$ imply that $u_i$ converges to $u_\infty := e^{\frac{(n-2)f}{2}}$ in $L^q(\Sph^n)$. Then Lemma \ref{lem-uniform-integrability} implies that $u_\infty \not\equiv 0$ in $L^{q}(\Sph^n)$. Therefore, by the dichotomy property of $u_\infty$ in Proposition \ref{prop-cutoff-infimum}, we conclude that the essential infimum of $u_\infty$ is positive.
\end{proof}


\subsection{A uniform lower bound on the conformal factors}\label{subsect: uniform lower bound}
As an application of the spherical mean inequality developed in section \ref{sec:Spherical Mean Method} and Proposition \ref{prop: positive lower bound}, we obtain a uniform lower bound on the sequence $u_i$.

Because in dimensions $n \geq 5$, the spherical mean inequality is established for the function $u^\frac{2}{n-2}$ instead of $u$, we first prove the following lemma.
\begin{lem}\label{lem: L1 convergence of a power of u}
Let $u_j$ be a sequence of smooth positive functions in $\Sph^n$, $n \geq 5$. Assume $u_j \to u_\infty$ in $L^1(\Sph^n)$ and $\inf\limits_{\Sph^n} u_\infty \geq e_\infty >0$, then $u^{\frac{2}{n-2}}_j \to u^{\frac{2}{n-2}}_\infty$ in $L^1(\Sph^n)$.
\end{lem}
\begin{proof}
We begin by calculating
\begin{align}
|u_{j}- u_{\infty}|&=\left|\sqrt{u_{j}^{2/(n-2)} }-\sqrt{u_{\infty}^{2/(n-2)}}\right|
\\\quad &\cdot\left| \sum_{i=0}^{n-2} \left(\sqrt{u_{j}^{2/(n-2)}   }\right)^{i} \left( \sqrt{u_{\infty}^{2/(n-2)}}\right)^{n-2-i}\right|,
\end{align}
where
\begin{align}
&  \sum_{i=0}^{n-2} \left(\sqrt{u_{j}^{2/(n-2)}}\right)^{i} \left(\sqrt{u_{\infty}^{2/(n-2)}}\right)^{n-2-i}  \\
 & \geq  \left(\sqrt{u_{j}^{2/(n-2)}}+\sqrt{u_{\infty}^{2/(n-2)}}\right) \left(\sqrt{u_{\infty}^{2/(n-2)}}\right)^{n-3} \\
 &\geq   e_{\infty}^{(n-3)/2} \left(\sqrt{u_{j}^{2/(n-2)}}+\sqrt{u_{\infty}^{2/(n-2)}}\right).
\end{align}
Combine these two inequalities, then we have
\begin{align}
\int_{\Sph^n}|u_{j}- u_{\infty}| dV_{g_{\Sph^n}} \geq e_{\infty}^{(n-3)/2} \int_{\Sph^n}\left| u_{j}^{2/(n-2)} -u_{\infty}^{2/(n-2)}\right| dV_{g_{\Sph^n}}. 
\end{align}
This finishes the proof.
\end{proof}

We now show that there must be a uniform positive lower bound for the sequence of conformal factors.

\begin{prop}\label{prop: uniform positive lower bound}
Let $u_i$ be a sequence of smooth positive functions in $\Sph^n$, $n \geq 3$. Assume that the metrics $ g_i := u_i^{\frac{4}{n-2}} g_{\Sph^n}$ have nonnegative scalar curvature, i.e. $\Sc_{g_i} \geq 0$, and uniformly bounded volumes, i.e.
\be
V^{-1} \leq \vol_{g_i}(\Sph^n) \leq V \text{ for all } \, i \in \N,
\ee
for some $V>0$. We also assume that the sequence $u_i$ satisfies the uniform integrability condition, in the sense that for some $\alpha \in (0,1)$, $\Lambda>0$
\be
\vol_{g_i}(U) \le  \Lambda \vol_{g_{\Sph^n}}(U)^{\alpha}, \quad \forall U \subset M \text{ measurable}.
\ee
Then there exists $i_0 \in \N$ such that $u_i \geq \frac{e_\infty}{4}>0$ in $\Sph^n$ holds for all $i \geq i_0$, where $e_\infty:= \inf\limits_{\Sph^n} u_\infty >0$.
\end{prop}
\begin{proof}
Recall that $\Sc_{g_i} \geq 0$ implies $\Delta_{g_{\Sph^n}} u_i \leq \frac{n(n-2)}{4} u_i$. 

Let's first prove the proposition in the case of $n \in \{3, 4\}$. In this case, Lemma \ref{lem: spherical mean inequality} implies that 
\begin{equation}
u_i(x) \geq \dashint_{\partial B_r(x)} u_i d\sigma - c(n) \|u_i\|_{L^{\frac{2n}{n-2}}(\Sph^n)}r
\end{equation}
holds for any $x\in \Sph^n$ and $0 < r < \frac{\pi}{2}$, where $c(n):=\frac{n(n-2)}{4}\left( \frac{\pi}{2} \right)^{\frac{n+2}{2n}} \left( \omega_{n-1}\right)^{\frac{2-n}{2n}}$ and $\omega_{n-1}$ is the volume of $\Sph^{n-1}$ with the standard metric $g_{\Sph^n}$. Because $\|u_i\|_{L^{\frac{2n}{n-2}}(\Sph^n)} = \left( \vol_{g_i}(\Sph^n) \right)^{\frac{n-2}{2n}} \leq V^{\frac{n-2}{2n}}$ for all $i \in \N$, we further obtain that
\begin{equation}
u_i(x) \geq \dashint_{\partial B_r(x)} u_i d\sigma - c(n)V^{\frac{n-2}{2n}}r
\end{equation}
holds for any $x \in \Sph^n$, $0 < r < \frac{\pi}{2}$, and all $i \in \N$. Multiplying the inequality by the area of $\partial B_r(x)$ with respect to $g_{\Sph^n}$, that is $\omega_{n-1}\sin^{n-1} r$, we then obtain
\begin{equation}
u_i(x) \omega_{n-1} \sin^{n-1}r 
\geq
\int_{\partial B_r(x)}u_i d\sigma - c(n) V^{\frac{n-2}{2n}} \omega_{n-1} r \sin^{n-1}r.
\end{equation}
Integrating the inequality and applying Proposition \ref{prop: positive lower bound} gives
\begin{align}
 u_i(x) \vol_{g_{\Sph^n}}(B_r(x)) 
& \geq 
\int_{B_{r}(x)}u_i dV_{g_{\Sph^n}} 
- c(n)V^{\frac{n-2}{2n}}\omega_{n-1} \int^r_0 r \sin^{n-1}s ds \\
& \geq  \int_{B_r(x)} u_\infty dV_{g_{\Sph^n}} - \|u_i - u_\infty \|_{L^1(\Sph^n)} \\
&  - c(n)V^{\frac{n-2}{2n}}\omega_{n-1} \int^r_0 s \sin^{n-1}s ds \\
& \geq  e_\infty \vol_{g_{\Sph^n}}(B_r(x)) - \|u_i - u_\infty\|_{L^1(\Sph^n)} \\
&  - c(n)V^{\frac{n-2}{2n}}\omega_{n-1} \int^r_0 s \sin^{n-1}s ds.
\end{align}
Dividing the inequality by $\vol_{g_{\Sph^n}}(B_r(x))$ produces that
\begin{equation}
u_i(x) \geq e_\infty - \frac{\|u_i - u_\infty\|_{L^1(\Sph^n)}}{\vol_{g_{\Sph^n}}(B_r(x))} - c(n)V^{\frac{n-2}{2n}} \frac{\int^r_0 s \sin^{n-1} s ds}{\int^r_0 \sin^{n-1}s ds}
\end{equation}
holds for any $x \in \Sph^n$, $0 < r < \frac{\pi}{2}$, and all $i$. Because 
\begin{equation}
\lim\limits_{r \to 0} \frac{\int^r_0 s \sin^{n-1} s ds}{\int^r_0 \sin^{n-1}s ds} =0,
\end{equation}
we can choose $0 < r_1 < \frac{\pi}{2}$ such that
\begin{equation}
c(n)V^{\frac{n-2}{2n}} \frac{\int^{r_1}_0 s \sin^{n-1} s ds}{\int^{r_1}_0 \sin^{n-1}s ds} \leq \frac{e_\infty}{4}.
\end{equation}
Moreover, because $u_i \to u_\infty$ in $L^1(\Sph^n)$ by Corollary \ref{cor:W12WeakConvergence}, there exists $i_0 \in \N$ such that for any $i > i_0$ 
\begin{equation}
\frac{\|u_i - u_\infty\|_{L^1(\Sph^n)}}{\vol_{g_{\Sph^n}}(B_{r_1}(x))} \leq \frac{e_\infty}{4}.
\end{equation}
As a result, we obtain that for any $i>i_0$
\begin{equation}
u_i(x) \geq e_\infty - \frac{\|u_i - u_\infty\|_{L^1(\Sph^n)}}{\vol_{g_{\Sph^n}}(B_{r_1}(x))} - c(n)V^{\frac{n-2}{2n}} \frac{\int^{r_1}_0 s \sin^{n-1} s ds}{\int^{r_1}_0 \sin^{n-1}s ds} \geq \frac{e_\infty}{4}
\end{equation}
holds for all $x \in \Sph^n$.

In the case of $n \geq 5$, we work on $u^{\frac{2}{n-2}}_i$ and $u^{\frac{2}{n-2}}_\infty$, instead of $u_i$ and $u_\infty$. By using the spherical mean inequality of $u^{\frac{2}{n-2}}_i$ in Proposition \ref{lem: spherical mean inequality high dim}, the positive infimum of $u_\infty$ in Proposition \ref{prop: positive lower bound}, and $L^1$-convergence of $u^{\frac{2}{n-2}}_i$ in Lemma \ref{lem: L1 convergence of a power of u}, applying the same argument as in the case of $n \in \{3, 4\}$ for $u^{\frac{2}{n-2}}_i$, we can obtain that there exists $i_0$ such that $u^{\frac{2}{n-2}}_i \geq \left( e_\infty \right)^{\frac{2}{n-2}}$ holds for all $i \geq i_0$. We omit the detailed argument, since it is verbatim as the first case, only replacing $u_i$ by $u^{\frac{2}{n-2}}_i$.
\end{proof}

As an application of the uniformly positive lower bound of $u_i$ in Proposition \ref{prop: uniform positive lower bound}, we prove the following convergence result for $f_i = \frac{2}{n-2} \ln u_i$, which will be used in the proof of Theorem \ref{thm-MainTheorem-sec6}.

\begin{cor}\label{cor: uniform L^2 bound on f_j}
Let $u_i$ be a sequence of smooth positive functions in $\Sph^n, n \geq 3$. Assume that the metrics $ g_i := u_i^{\frac{4}{n-2}} g_{\Sph^n}$ have nonnegative scalar curvature, i.e. $\Sc_{g_i} \geq 0$, and uniformly bounded volumes, i.e.
\be
V^{-1} \leq \vol_{g_i}(\Sph^n) \leq V \text{ for all } \, i \in \N,
\ee
for some $V>0$. We also assume that the sequence $u_i$ satisfies the uniform integrability condition, in the sense that for some $\alpha \in (0,1)$, $\Lambda>0$
\be
\vol_{g_i}(U) \le  \Lambda \vol_{g_{\Sph^n}}(U)^{\alpha}, \quad \forall U \subset M \text{ measurable}.
\ee
Then for $f_i = \frac{2}{n-2}\ln(u_i)$ we find that
\begin{align}
    \int_{\Sph^n}f_i^2dV_{g_{\Sph^n}} \le C,
\end{align}
and $f_i \rightarrow f_{\infty}$ in $L^p$, $p \in \left[1,\frac{2n}{n-2}\right)$. Moreover, $f_\infty: \Sph^n \to (-\infty, +\infty]$ is defined everywhere in $\Sph^n$ and is lower semicontinuous.
\end{cor}

\begin{proof}
    By Jensen's inequality we know
    \begin{align}
        \dashint_{\Sph^n}f_i dV_{g_{\Sph^n}}&=\frac{1}{n}\dashint_{\Sph^n}\ln\left(u_i^{\frac{2n}{n-2}}\right) dV_{g_{\Sph^n}}
        \\&\le \ln\left(\dashint_{\Sph^n}u_i^{\frac{2n}{n-2}}dV_{g_{\Sph^n}}\right) \le \ln\left(\frac{V}{\vol_{g_{\Sph^n}}(\Sph^n)}\right).
    \end{align}
   Then by Proposition \ref{prop: uniform positive lower bound} we know $f_i \ge \frac{\ln(e_{\infty}/4)}{n-2}$ so that
    \begin{align}
       \int_{\Sph^n}f_i dV_{g_{\Sph^n}}\ge\int_{\{f_i\le 0\}}f_i dV_{g_{\Sph^n}} \ge \Vol(\Sph^n)\min\left\{\frac{\ln(e_{\infty}/4)}{n-2},0\right\}.
    \end{align}
    This shows that $\bar{f}_i = \dashint_{\Sph^n}f_i dV_{g_{\Sph^n}}$ is uniformly bounded from above and below. Now by the Poincar\'e inequality and Lemma \ref{lem: gradient L2 bound of f} we have
    \begin{align}
        \int_{\Sph^n}|f_i-\bar{f_i}|^2dV_{g_{\Sph^n}} \le C_P  \int_{\Sph^n}|\nabla f_i|^2dV_{g_{\Sph^n}} \le C',
    \end{align}
    and hence
    \begin{align}
        \int_{\Sph^n} f_i^2dV_{g_{\Sph^n}} &= \int_{\Sph^n} (f_i-\bar{f}_i+\bar{f}_i)^2dV_{g_{\Sph^n}} 
        \\&= \int_{\Sph^n} ((f_i-\bar{f}_i)+\bar{f}_i)^2dV_{g_{\Sph^n}} 
        \le 2 \int_{\Sph^n}|f_i-\bar{f_i}|^2 + \bar{f_i}^2dV_{g_{\Sph^n}} \le C.
    \end{align}
    Then the convergence of $f_i$ follows from Lemma \ref{lem: gradient L2 bound of f} and the Rellich-Kondrachov compactness theorem. The last observation for the limit function $f_\infty$ follows from Proposition \ref{prop: lower semicontinuous f}, since $\Sc_{g_i} \geq 0$ implies $\Delta f_i \leq \frac{n}{2}$.
\end{proof}

\subsection{Lower semi-continuity of limiting conformal factor}\label{subsect: lower semi-continuity}
As another application of the $L^1$ convergence of $u^{\frac{2}{n-2}}_i$ obtained in Lemma \ref{lem: L1 convergence of a power of u}, with the help of ball average inequality for $u^{\frac{2}{n-2}}$ on $\Sph^n$, $ n \geq 5$, in Lemma \ref{lem: ball average inequality high dim}, we prove the lower semi-continuity of $u_\infty$ in the dimension of $n \geq 5$. Combining this with Proposition \ref{prop: lower semicontinuous u} gives the lower semi-continuity of $u_\infty$ in the dimension of $n \geq 3$ stated in Theorem \ref{thm-MainTheorem1}.
\begin{prop}\label{prop: lower semi-continuity of u high dim}
 Let $V, \Lambda > 0$, $ n\geq 3$,  $(\Sph^n,g_{\Sph^n})$ be the standard round sphere, and $f_j:\Sph^n\rightarrow \R$ a sequence of smooth functions defining conformal metrics $g_j=e^{2f_j}g_{\Sph^n}$. If 
    \begin{align}
        \Sc_{g_j} &\ge 0, \qquad V^{-1} \le \Vol_{g_j}(\Sph^n) \le V, \qquad
        \\\Vol_{g_j}(U) &\le  \Lambda \Vol_{g_{\Sph^n}}(U)^{\alpha}, \alpha \in (0,1), \quad \forall U \subset \Sph^n \text{ measurable}, \label{UnifIntegrabilityVolume}
    \end{align}
then the limit function $u_\infty := e^{\frac{(n-2)f_\infty}{2}}$ obtained in Corollary \ref{cor:W12WeakConvergence} is well defined everywhere in $\Sph^n$ with value in $(0, \infty]$, and lower semi-continuous. 
\end{prop}
\begin{proof}
In the case of $3 \leq n \leq 4$, the conclusion has been proved in Proposition \ref{prop: lower semicontinuous u}.
In the case of $n\geq 5$, by the same argument as in the proof of Proposition \ref{prop: lower semicontinuous u}, and applying Lemmas \ref{lem: ball average inequality high dim} and \ref{lem: L1 convergence of a power of u}, one can obtain that $u^{\frac{2}{n-2}}_\infty$ is well defined everywhere in $\Sph^n$ with value in $(0, \infty]$ and lower semi-continuous. Then the desired property of $u_\infty$ directly follows. 
\end{proof}


\section{Singular Set Decomposition} \label{sec: Singular Set Decomposition}

In this section we prove the portion of Theorem \ref{thm-MainTheorem1} which pertains to the choice of a good set $Z_j$ with good volume controls. Since we will only prove a portion of Theorem \ref{thm-MainTheorem1} we restate that portion of the theorem here.

\begin{thm}\label{thm-MainTheorem-sec6}
    Let $V, \Lambda >0$, $(\Sph^n,g_{\Sph^n})$ be the round sphere, $n \geq 3$, and $f_j:\Sph^n\rightarrow \R$ a sequence of functions defining conformal metrics $g_j=e^{2f_j}g_{\Sph^n}$. Then if 
    \begin{equation}
        \Sc_{g_j} \ge 0, \qquad V^{-1} \le \Vol_{g_j}(M) \le V,
    \end{equation}
    \begin{equation}\label{eqn: uniform-integrability-singular-set}
    \Vol_{g_j}(U) \le  \Lambda \Vol_{g_{\Sph^n}}(U)^{\alpha}, \alpha >0, \quad \forall U \subset \Sph^n \text{ measurable},
    \end{equation}
    then there exists a subsequence, and a limiting lower semicontinuous function $f_{\infty} \in W^{1,2}(\Sph^n)$ so that there is a measurable set $Z_j \subset \Sph^n$ where $\Area_{g_{\Sph^n}}(\partial Z_j)+\Vol_{g_{\Sph^n}}( Z_j)\rightarrow 0$, $\Vol_{g_j}(\Sph^n \setminus Z_j) \rightarrow \Vol_{g_{\infty}}(\Sph^n)$, and $|f_j-f_{\infty}|\le C_j$ on $\Sph^n \setminus Z_j$, $C_j \searrow 0$. 
\end{thm}

\begin{proof}

Consider the following good set
\begin{align}
    E_j^{\tau}&=\{x \in \Sph^n:|f_j(x)-f_{\infty}(x)|^2\le \tau\},
\end{align}
By taking advantage of the Coarea formula for BV functions (see Theorem 5.9 of \cite{Evans-Gariepy}), in a similar way as pioneered by C. Dong \cite{Dong-PMT_Stability}, we find for each $\varepsilon_2>\varepsilon_1>0$ that
\begin{align}
    \int_{\varepsilon_1}^{\varepsilon_2}& \mathcal{H}^2_{g_{\Sph^n}}\left(\partial^*(E_j^{\tau})^c \right)d\tau
    \\&= \int_{E^{\varepsilon_2}_j\setminus E^{\varepsilon_1}_j} |\nabla |f_j-f_{\infty}|^2|dV_{g_{\Sph^n}}
    \\&\le 2\int_{\Sph^n}  |f_j-f_{\infty}||\nabla f_j-\nabla f_{\infty}||dV_{g_{\Sph^n}}
    \\&\le 2\left(\int_{\Sph^n}  |f_j-f_{\infty}|^2dV_{g_{\Sph^n}} \right)^{1/2}\left(\int_{\Sph^n}|\nabla f_j-\nabla f_{\infty}|^2dV_{g_{\Sph^n}}\right)^{1/2} 
    \\&\le 2\left(\int_{\Sph^n}  |f_j-f_{\infty}|^2dV_{g_{\Sph^n}} \right)^{1/2}\left(\int_{\Sph^n}|\nabla f_j|^2+|\nabla f_{\infty}|^2dV_{g_{\Sph^n}}\right)^{1/2} 
    \\&\le C(j),
\end{align}
where $\partial ^*$ denotes the reduced boundary (see chapter 5 of \cite{Evans-Gariepy}) and $C(j) \rightarrow 0$ as $j\rightarrow \infty$ and we have applied Lemma \ref{lem: gradient L2 bound of f} and Corollary \ref{cor: uniform L^2 bound on f_j} in the last line.

Then if we consider $\varepsilon_j=\varepsilon_{2,j}-\varepsilon_{1,j}>0$ we can conclude that there is some $\tau_j \in [\varepsilon_{1,j},\varepsilon_{2,j}]$ so that
\begin{align}
 \mathcal{H}^2_{\Sph^n}\left(\partial^*(E_j^{\tau_j})^c\right)  \le \frac{C(j)}{\varepsilon_j}.
\end{align}
Now if we pick $\varepsilon_j=\sqrt{C(j)}$ then for $\varepsilon_{1,j}=\frac{\sqrt{C(j)}}{2} \le \tau_j \le \sqrt{C(j)}=\varepsilon_{2,j}$ we see that
\begin{align}
 \mathcal{H}^2_{\Sph^n}\left(\partial^*(E_j^{\tau_j})^c\right)  \le 2\sqrt{C(j)}.
\end{align}
Hence if we choose $Z_j=(E^{\tau_j}_j)^c$ we see that $\Area_{g_{\Sph^n}}(\partial^* Z_j)\rightarrow 0$.

Now we would like to show that the volume of $Z_j$ is going to $0$ with respect to $g_{\Sph^n}$. If we apply Chebyshev's inequality we find
\begin{align}
    \Vol_{g_{\Sph^n}}((E_j^{\tau_j})^c)&=\Vol_{g_{\Sph^n}}(\{x\in \Sph^n: |f_j-f_{\infty}|\ge \tau_j\}) 
    \\&\le \frac{1}{\tau_j^2} \int_{\{|f_j-f_{\infty}|\ge \tau_j\}} |f_j-f_{\infty}|^2 dV_{g_{\Sph^n}}
    \\&\le \frac{1}{\tau_j^2} \int_{\Sph^n} |f_j-f_{\infty}|^2 dV_{g_{\Sph^n}} \le C'\frac{C(j)^2}{\tau_j^2} = C'C(j),
\end{align}
and hence $\Vol_{g_{\Sph^n}}((E_j^{\tau_j})^c) \rightarrow 0$ as $j \rightarrow \infty$. This implies that
\begin{align}
    \Vol_{g_{\Sph^n}}(E_j^{\tau_j}) \rightarrow \Vol_{g_{\Sph^n}}(\Sph^n).
\end{align}

By the assumption of uniform integrability of $e^{nf_j}$ in (\ref{eqn: uniform-integrability-singular-set}), we see that there exists a subsequence so that
\begin{align}
    e^{nf_j} \rightharpoonup e^{nf_{\infty}},
\end{align}
in $L^1$. Hence we find that $\Vol_{g_j}(\Sph^n) \rightarrow \Vol_{g_{\infty}}(\Sph^n)$ as well as
\begin{align}
    \Vol_{g_j}(E_j^{\tau_j})&=\int_{E_j^{\tau_j}}e^{nf_j}dV_{g_{\Sph^n}}
    \\&=\int_{\Sph^n}e^{nf_j}dV_{g_{\Sph^n}}-\int_{(E_j^{\tau_j})^c}e^{nf_j}dV_{g_{\Sph^n}},
\end{align}
where we know that
\begin{align}
\int_{(E_j^{\tau_j})^c}e^{nf_j}dV_{g_{\Sph^n}} \le C\Vol_{g_{\Sph^n}}((E_j^{\tau_j})^c)^{\alpha},
\end{align}
by the assumption (\ref{eqn: uniform-integrability-singular-set}).
Hence we see that 
\begin{align}
  \Vol_{g_j}(E_j^{\tau_j})\rightarrow   \Vol_{g_{\infty}}(\Sph^n).
\end{align}
\end{proof}

\section{Total Scalar Curvature Upper Bound}\label{sec: Total Scalar Curvature Upper Bound}
In this section we investigate the consequences of adding the assumption of bounded total scalar curvature by proving the portion of Theorem \ref{thm-MainTheorem2} pertaining to the singular set, as stated below.

\begin{thm}\label{thm-MainTheorem-sec7}
    Let $V, \Lambda >0$, $(\Sph^n,g_{\Sph^n})$ be the standard round sphere, $n \geq 3$, and $f_j:\Sph^n\rightarrow \R$ a sequence of functions defining conformal metrics $g_j=e^{2f_j}g_{\Sph^n}$. Then if 
    \begin{align}
        \Sc_{g_j} &\ge 0, \qquad V^{-1} \le \Vol(M,g_j) \le V, \qquad
        \\\Vol_{g_j}(U) &\le  \Lambda \Vol_{g_{\Sph^n}}(U)^{\alpha}, \alpha >0, \quad \forall U \subset \Sph^n \text{ measurable},
    \end{align}
     and a uniform bound on the scalar curvature
    \begin{align}
        \int_{\Sph^n} \Sc_{g_j} dV_{g_j} \le R_0,
    \end{align}
    then one obtains weak $W^{1,2}$ convergence of $e^{\frac{(n-2)f_j}{2}}$ to $e^{\frac{(n-2)f_{\infty}}{2}}$ on a subsequence.  Furthermore, there is a measurable set $Z_j \subset \Sph^n$ so that $\Area_{g_{\Sph^n}}(\partial Z_j)+\Vol_{g_{\Sph^n}}(Z_j)\rightarrow 0$, $\Vol_{g_j}(\Sph^n \setminus Z_j) \rightarrow \Vol_{g_{\infty}}(\Sph^n)$, and $\left|e^{\frac{(n-2)f_j}{2}}-e^{\frac{(n-2)f_{\infty}}{2}}\right|^2 \le C_j$, $C_j \searrow 0$, on $\Sph^n \setminus Z_j$, and $(\Sph^n,e^{2f_{\infty}}g_{\Sph^n})$ has weak positive scalar curvature in the sense of Definition \ref{def-Weak Positive Scalar Confromal}.
\end{thm}
\begin{proof}
 By Lemma \ref{lem:ScalarFormulas} we see that
 \begin{align}
         \int_{\Sph^n} &
        \Sc_{g_j} dV_{g_j} \label{Eq-Total Scalar Curvature} \\&=\int_{\Sph^n}e^{nf_j}R_{g_j}dV_{g_{\Sph^n}} 
        \\&=  \int_{\Sph^n}e^{(n-2)f_j}\left(R_{g_0}-2(n-1)\Delta f_j - (n-2)(n-1)|\nabla f_j|^2\right)dV_{g_{\Sph^n}} 
        \\&=  \int_{\Sph^n}n(n-1)e^{(n-2)f_j} + (n-2)(n-1)e^{(n-2)f_j}|\nabla f_j|^2dV_{g_{\Sph^n}}
        \\&=  \int_{\Sph^n}n(n-1)e^{(n-2)f_j} + \frac{4(n-1)}{n-2}|\nabla e^{\frac{(n-2)f_j}{2}}|^2dV_{g_{\Sph^n}}, \label{Eq-Total Scalar Conformal Factor}
    \end{align}
    and hence the $W^{1,2}$ norm of $e^{\frac{(n-2)f_j}{2}}$ is bounded and we obtain subsequential convergence weakly in $W^{1,2}$ to $e^{\frac{(n-2)f_{\infty}}{2}} \in W^{1,2}(\Sph^n)$. 
    
We also notice that for all non-negative $\varphi \in W^{1,2}(\Sph^n)$ that
\begin{equation}
    -\int_{\Sph^n} g\left(\nabla \varphi, \nabla e^{\frac{(n-2)f_j}{2}}\right) dV_{g_{\Sph^n}} \le \frac{n(n-2)}{4} \int_{\Sph^n} \varphi e^{\frac{(n-2)f_j}{2}}dV_{g_{\Sph^n}},
\end{equation}
and hence by the weak convergence we have shown we know that
\begin{equation}\label{eq: Limit Weakly Solved Good PDE}
    -\int_{\Sph^n} g\left(\nabla \varphi, \nabla e^{\frac{(n-2)f_{\infty}}{2}}\right) dV_{g_{\Sph^n}} \le \frac{n(n-2)}{4} \int_{\Sph^n} \varphi e^{\frac{(n-2)f_{\infty}}{2}}dV_{g_{\Sph^n}} ,
\end{equation}
which means that $e^{\frac{(n-2)f_{\infty}}{2}}$ weakly solves the equation $\Delta u \le C u$.

    Now we define the following good set
\begin{equation}
E^{\tau}=\{x \in \Sph^n:|e^{\frac{(n-2)f_j(x)}{2}}-e^{\frac{(n-2)f_{\infty}(x)}{2}}|^2\le \tau\}.
\end{equation}
By taking advantage of the Coarea formula in the same way as pioneered by C. Dong \cite{Dong-PMT_Stability} we find for each $\varepsilon_2>\varepsilon_1>0$ that
\begin{align}
    \int_{\varepsilon_1}^{\varepsilon_2}& \mathcal{H}^2_{\Sph^n}\left(\partial^*(E_j^{\tau})^c\right)d\tau
    \\&=  \int_{E^{\varepsilon_2}_j\setminus E^{\varepsilon_1}_j} |\nabla |e^{\frac{(n-2)f_j(x)}{2}}-e^{\frac{(n-2)f_{\infty}(x)}{2}}|^2|dV_{g_{\Sph^n}}
    \\&\le  2\int_{\Sph^n}  |e^{\frac{(n-2)f_j(x)}{2}}-e^{\frac{(n-2)f_{\infty}(x)}{2}}||\nabla e^{\frac{(n-2)f_j(x)}{2}}-\nabla e^{\frac{(n-2)f_{\infty}(x)}{2}}||dV_{g_{\Sph^n}}
    \\&\le  2\left(\int_{\Sph^n}  |e^{\frac{(n-2)f_j(x)}{2}}-e^{\frac{(n-2)f_{\infty}(x)}{2}}|^2dV_{g_{\Sph^n}} \right)^{1/2}
    \\&  \cdot \left(\int_{\Sph^n}|\nabla e^{\frac{(n-2)f_j(x)}{2}}-\nabla e^{\frac{(n-2)f_{\infty}(x)}{2}}|^2dV_{g_{\Sph^n}}\right)^{1/2} 
    \\&\le  2\left(\int_{\Sph^n}  |e^{\frac{(n-2)f_j(x)}{2}}-e^{\frac{(n-2)f_{\infty}(x)}{2}}|^2dV_{g_{\Sph^n}} \right)^{1/2}
    \\&  \cdot \left(\int_{\Sph^n}|\nabla e^{\frac{(n-2)f_j(x)}{2}}|^2+|\nabla e^{\frac{(n-2)f_{\infty}(x)}{2}}|^2dV_{g_{\Sph^n}}\right)^{1/2} \le C(j),
\end{align}
where where $\partial ^*$ denotes the reduced boundary (see chapter 5 of \cite{Evans-Gariepy}) and $C(j) \rightarrow 0$ as $j\rightarrow \infty$ by Corollary \ref{cor:W12WeakConvergence} and the observation above. 

Then if we consider $\varepsilon_j=\varepsilon_{2,j}-\varepsilon_{1,j}>0$ we can conclude that there is some $\tau_j \in [\varepsilon_{1,j},\varepsilon_{2,j}]$ so that
\begin{equation}
 \mathcal{H}^2_{\Sph^n}\left(\partial^*(E_j^{\tau_j})^c\right)  \le \frac{C(j)}{\varepsilon_j},
\end{equation}
Now if we pick $\varepsilon_j=\sqrt{C(j)}$ then for $\varepsilon_{1,j}=\frac{\sqrt{C(j)}}{2} \le \tau_j \le \sqrt{C(j)}=\varepsilon_{2,j}$ we see that
\begin{align}
 \mathcal{H}^2_{\Sph^n}\left(\partial^*(E_j^{\tau_j})^c\right)  \le 2\sqrt{C(j)}.
\end{align}
Hence if we choose $Z_j=(E^{\tau_j}_j)^c$ we see that $\Area_{g_{\Sph^n}}(\partial^* Z_j)\rightarrow 0$.

Now we would like to show that the volume of $Z_j$ is going to $0$ with respect to $g_{\Sph^n}$. If we apply Chebyshev's inequality we find
\begin{align}
    &\Vol_{g_{\Sph^n}}((E_j^{\tau_j})^c)
    \\&=\Vol_{g_{\Sph^n}}(\{x\in \Sph^n: |e^{\frac{(n-2)f_j(x)}{2}}-e^{\frac{(n-2)f_{\infty}(x)}{2}} |\ge \tau_j\}) 
    \\&\le \frac{1}{\tau_j^2} \int_{\left\{|e^{\frac{(n-2)f_j(x)}{2}}-e^{\frac{(n-2)f_{\infty}(x)}{2}} |\ge \tau_j\right\}} |e^{\frac{(n-2)f_j(x)}{2}}-e^{\frac{(n-2)f_{\infty}(x)}{2}}|^2 dV_{g_{\Sph^n}}
    \\&\le \frac{1}{\tau_j^2} \int_{\Sph^n} |e^{\frac{(n-2)f_j(x)}{2}}-e^{\frac{(n-2)f_{\infty}(x)}{2}}|^2 dV_{g_{\Sph^n}} \le C'\frac{C(j)^2}{\tau_j^2} = C'C(j),
\end{align}
and hence $\Vol_{g_{\Sph^n}}((E_j^{\tau_j})^c) \rightarrow 0$ as $j \rightarrow \infty$. This implies that
\begin{align}
    \Vol_{g_{\Sph^n}}(E_j^{\tau_j}) \rightarrow \Vol_{g_{\Sph^n}}(\Sph^n).
\end{align}

By the assumption of uniform integrability of $e^{nf_j}$ we see that there exists a subsequence so that
\begin{align}
    e^{nf_j} \rightharpoonup e^{nf_{\infty}},
\end{align}
in $L^1$. Hence we find that $\Vol_{g_j}(\Sph^n) \rightarrow \Vol_{g_{\infty}}(\Sph^n)$ as well as
\begin{align}
    \Vol_{g_j}(E_j^{\tau_j})&=\int_{E_j^{\tau_j}}e^{nf_j}dV_{g_{\Sph^n}}
    \\&=\int_{\Sph^n}e^{nf_j}dV_{g_{\Sph^n}}-\int_{(E_j^{\tau_j})^c}e^{nf_j}dV_{g_{\Sph^n}},
\end{align}
where we know that
\begin{align}
    \int_{(E_j^{\tau_j})^c}e^{nf_j}dV_{g_{\Sph^n}} \le \Lambda \left( \Vol_{g_{\Sph^n}}\left( (E_j^{\tau_j})^c \right) \right)^{\alpha}.
\end{align}
Hence we see that 
\begin{align}
  \Vol_{g_j}(E_j^{\tau_j})\rightarrow   \Vol_{g_{\infty}}(\Sph^n).
\end{align}
\end{proof}

    \bibliographystyle{plain}
    \bibliography{bibliography}

\begin{bibdiv}
\begin{biblist}

\bib{ABKLLarull}{misc}{
      author={Allen, Brian},
      author={Bryden, Edward},
      author={Kazaras, Demetre},
       title={On the stability of {L}larull's theorem in dimension three},
        date={2023},
        note={arXiv:2305.18567},
}

\bib{Allen-conf-torus}{article}{
      author={Allen, Brian},
       title={Almost non-negative scalar curvature on {R}iemannian manifolds
  conformal to tori},
        date={2021},
        ISSN={1050-6926,1559-002X},
     journal={J. Geom. Anal.},
      volume={31},
      number={11},
       pages={11190\ndash 11213},
}

\bib{AS-relating}{article}{
      author={Allen, Brian},
      author={Sormani, Christina},
       title={Relating notions of convergence in geometric analysis},
        date={2020},
        ISSN={0362-546X},
     journal={Nonlinear Anal.},
      volume={200},
       pages={111993, 33},
         url={https://doi-org.ezproxy.gc.cuny.edu/10.1016/j.na.2020.111993},
      review={\MR{4108435}},
}

\bib{Brezis}{book}{
      author={Brezis, Haim},
       title={Functional analysis, sobolev spaces and partial differential
  equations},
   publisher={Springer},
        date={2010},
}

\bib{Chang-Yang}{article}{
      author={Chang, Sun-Yung~A.},
      author={Yang, Paul~C.},
       title={Compactness of isospectral conformal metrics on {$S^3$}},
        date={1989},
        ISSN={0010-2571,1420-8946},
     journal={Comment. Math. Helv.},
      volume={64},
      number={3},
       pages={363\ndash 374},
}

\bib{Dong-Li}{article}{
      author={Dong, Conghan},
      author={Li, Yuxiang},
       title={Compactness of conformal metrics with integral bounds on {R}icci
  curvature},
        date={2022},
        ISSN={0030-8730,1945-5844},
     journal={Pacific J. Math.},
      volume={316},
      number={1},
       pages={65\ndash 79},
}

\bib{DLX}{article}{
      author={Dong, Conghan},
      author={Li, Yuxiang},
      author={Xu, Ke},
       title={{$W^{1,p}$}-metrics and conformal metrics with
  {$L^{n/2}$}-bounded scalar curvature},
        date={2023},
        ISSN={0219-1997,1793-6683},
     journal={Commun. Contemp. Math.},
      volume={25},
      number={1},
       pages={Paper No. 2150047, 22},
}

\bib{Dong-PMT_Stability}{misc}{
      author={Dong, Conghan},
       title={Stability of the positive mass theorem in dimension three},
   publisher={arXiv},
        date={2022},
         url={https://arxiv.org/abs/2211.06730},
        note={arXiv:2211.06730},
}

\bib{dong2024stability}{misc}{
      author={Dong, Conghan},
       title={Stability for the 3{D} {R}iemannian {P}enrose inequality},
        date={2024},
        note={arXiv:2402.10299},
}

\bib{Dong-Song}{misc}{
      author={Dong, Conghan},
      author={Song, Antoine},
       title={Stability of {E}uclidean 3-space for the positive mass theorem},
        date={2023},
        note={arXiv:2302.07414},
}

\bib{Evans-Gariepy}{book}{
      author={Evans, Lawrence~C.},
      author={Gariepy, Ronald~F.},
       title={Measure theory and fine properties of functions},
     edition={Revised},
      series={Textbooks in Mathematics},
   publisher={CRC Press, Boca Raton, FL},
        date={2015},
        ISBN={978-1-4822-4238-6},
      review={\MR{3409135}},
}

\bib{Gromov-Plateau}{article}{
      author={Gromov, Misha},
       title={Plateau-{S}tein manifolds},
        date={2014},
        ISSN={1895-1074},
     journal={Cent. Eur. J. Math.},
      volume={12},
      number={7},
       pages={923\ndash 951},
         url={https://doi-org.ezproxy.gc.cuny.edu/10.2478/s11533-013-0387-5},
      review={\MR{3188456}},
}

\bib{Gilbarg-Trudinger}{book}{
      author={Gilbarg, David},
      author={Trudinger, Neil~S.},
       title={Elliptic partial differential equations of second order},
      series={Grundlehren der Mathematischen Wissenschaften},
   publisher={Springer-Verlag, Berlin-New York},
        date={1977},
      volume={Vol. 224},
}

\bib{Gursky}{article}{
      author={Gursky, Matthew~J.},
       title={Compactness of conformal metrics with integral bounds on
  curvature},
        date={1993},
        ISSN={0012-7094,1547-7398},
     journal={Duke Math. J.},
      volume={72},
      number={2},
       pages={339\ndash 367},
}

\bib{HZ}{misc}{
      author={Hirsch, Sven},
      author={Zhang, Yiyue},
       title={Stability of {L}larull's theorem in all dimensions},
        date={2023},
        note={arXiv:2310.14412},
}

\bib{KK}{misc}{
      author={Kazaras, Demetre},
      author={Xu, Kai},
       title={Drawstrings and flexibility in the {G}eorch conjecture},
        date={2023},
        note={arXiv:2309.03756},
}

\bib{Lee-LeFloch}{article}{
      author={Lee, Dan~A.},
      author={LeFloch, Philippe~G.},
       title={The positive mass theorem for manifolds with distributional
  curvature},
        date={2015},
        ISSN={0010-3616},
     journal={Comm. Math. Phys.},
      volume={339},
      number={1},
       pages={99\ndash 120},
         url={https://doi-org.ezproxy.gc.cuny.edu/10.1007/s00220-015-2414-9},
      review={\MR{3366052}},
}

\bib{Li-Zhou}{article}{
      author={Li, Yuxiang},
      author={Zhou, Zhipeng},
       title={Conformal metric sequences with integral-bounded scalar
  curvature},
        date={2020},
        ISSN={0025-5874,1432-1823},
     journal={Math. Z.},
      volume={295},
      number={3-4},
       pages={1443\ndash 1473},
}

\bib{Park-Tian-Wang-18}{article}{
      author={Park, Jiewon},
      author={Tian, Wenchuan},
      author={Wang, Changliang},
       title={A compactness theorem for rotationally symmetric {R}iemannian
  manifolds with positive scalar curvature},
        date={2018},
        ISSN={1558-8599},
     journal={Pure Appl. Math. Q.},
      volume={14},
      number={3-4},
       pages={529\ndash 561},
         url={https://doi-org.ezproxy.gc.cuny.edu/10.4310/pamq.2018.v14.n3.a5},
      review={\MR{4047408}},
}

\bib{SormaniIAS}{incollection}{
      author={Sormani, Christina},
       title={Conjectures on convergence and scalar curvature},
        date={2023},
   booktitle={Chapter for perspectives in scalar curvature. {V}ol. {2}},
      editor={Gromov, Mikhail~L.},
      editor={Lawson~Jr., H.~Blaine},
   publisher={World Scientific},
}

\bib{sormani2023extreme}{article}{
      author={Sormani, Christina},
      author={Tian, Wenchuan},
      author={Wang, Changliang},
       title={An extreme limit with nonnegative scalar curvature},
        date={2024},
        ISSN={0362-546X,1873-5215},
     journal={Nonlinear Anal.},
      volume={239},
       pages={Paper No. 113427, 24},
         url={https://doi.org/10.1016/j.na.2023.113427},
      review={\MR{4661651}},
}

\bib{SormaniOberwolfach}{misc}{
      author={Sormani, Christina},
      author={Tian, Wenchuan},
      author={Wang, Changliang},
       title={Oberwolfach workshop report: Analysis, geometry and topology of
  positive scalar curvature metrcs: Limits of sequences of manifolds with
  nonnegative scalar curvature and other hypotheses},
        date={2024},
        note={arXiv:2404.17121},
}

\bib{tian2023compactness}{article}{
      author={Tian, Wenchuan},
      author={Wang, Changliang},
       title={Compactness of sequences of warped product circles over spheres
  with nonnegative scalar curvature},
        date={2024},
     journal={Math. Ann. https://doi.org/10.1007/s00208-024-02816-w},
}

\end{biblist}
\end{bibdiv}
\end{document}